\documentclass[11pt, reqno]{amsart}
\usepackage{lmodern}
\usepackage{amsmath, amsthm, amssymb, amsfonts}
\usepackage[normalem]{ulem}
\usepackage{hyperref}

\usepackage{extarrows}

\usepackage{verbatim}
\usepackage{caption}
\usepackage{mathtools}

\usepackage{tikz-cd}
\theoremstyle{plain}
\newtheorem{thm}{Theorem}[section]
\newtheorem{cor}[thm]{Corollary}

\newtheorem{lemma}[thm]{Lemma}
\newtheorem{prop}[thm]{Proposition}

\newtheorem{Example}[thm]{Example}

\theoremstyle{definition}

\theoremstyle{remark}
\newtheorem{rmk}[thm]{Remark}

\newcommand{\BC}{{\mathbb{C}}}

\newcommand{\BP}{{\mathbb{P}}}
\newcommand{\BQ}{{\mathbb{Q}}}

\newcommand{\BZ}{{\mathbb{Z}}}

\newcommand{\CC}{{\mathcal C}}

\newcommand{\CE}{{\mathcal E}}
\newcommand{\CF}{{\mathcal F}}

\newcommand{\CH}{{\mathcal H}}
\newcommand{\CI}{{\mathcal I}}

\newcommand{\CL}{{\mathcal L}}
\newcommand{\CM}{{\mathcal M}}

\newcommand{\CO}{{\mathcal O}}
\newcommand{\CP}{{\mathcal P}}

\newcommand{\CX}{{\mathcal X}}

\DeclareFontFamily{OT1}{rsfs}{}
\DeclareFontShape{OT1}{rsfs}{n}{it}{<-> rsfs10}{}
\DeclareMathAlphabet{\curly}{OT1}{rsfs}{n}{it}

\begin{document}
\title[Perverse filtrations, Hilbert schemes, and $P=W$]
{Perverse filtrations, Hilbert schemes, and the $P=W$ Conjecture for parabolic Higgs bundles}
\date{\today}

\author{Junliang Shen}
\address{MIT, Department of Mathematics}
\email{jlshen@mit.edu}

\author{Zili Zhang}
\address{University of Michigan, Department of Mathematics}
\email{ziliz@umich.edu}

\begin{abstract}
We prove de Cataldo--Hausel--Migliorini's $P=W$ conjecture in arbitrary rank for parabolic Higgs bundles labeled by the affine Dynkin diagrams $\tilde{A}_0$, $\tilde{D}_4$, $\tilde{E}_6$, $\tilde{E}_7$, and $\tilde{E}_8$. Our proof relies on the study of the tautological classes on the Hilbert scheme of points on an elliptic surface with respect to the perverse filtration.
\end{abstract}
\baselineskip=14.5pt
%\baselineskip=12.5pt
%Standard baseline skip: 1.2*fontsize; see \the\baselineskip
\maketitle

\setcounter{tocdepth}{1} 

\tableofcontents
\setcounter{section}{-1}

\section{Introduction}

\subsection{Perverse filtrations}
Let $D_c^b(Y)$ denote the bounded derived category of constructible sheaves on a nonsingular algebraic variety $Y$. The truncation functor
\[
^\mathfrak{p}\tau_{\le k}: D_c^b(Y) \rightarrow D_c^b(Y)
\]
with respect to the perverse $t$-structure on $D_c^b(Y)$ (see \cite{BBD}) induces a natural morphism
\begin{equation}\label{trun0}
^\mathfrak{p}\tau_{\le k}\CC \rightarrow \CC
\end{equation}
for any object $\CC \in D_c^b(Y)$ and $k \in \BZ$.

Let $f:  X\to Y$ be a proper morphism between two irreducible nonsingular algebraic varieties of relative dimension $r$. The morphism
\[
^\mathfrak{p}\tau_{\le k}Rf_\ast \BQ_X \rightarrow Rf_\ast \BQ_X
\]
given by (\ref{trun0}) induces a morphism of (hyper-)cohomology groups,
\begin{equation}\label{perv_filt}
H^{d-(\mathrm{dim}X -r)}\Big{(}Y, ~~~~~ ^\mathfrak{p}\tau_{\le k}(Rf_\ast \BQ_X[\mathrm{dim}X -r]) \Big{)} \rightarrow H^d(X, \BQ).
\end{equation}
Following \cite{dCHM1}, we define $P^f_kH^d(X, \BQ) \subset H^d(X, \BQ)$ to be the image of (\ref{perv_filt}),\footnote{Here the shift $[\mathrm{dim}X-r]$ is to ensure that the perverse filtration is concentrated in the degrees $[0, 2r]$; see Section 1.} and we call the increasing filtration 
\begin{equation}\label{balabalah}
    P^f_0H^d(X,\BQ) \subset P^f_1H^d(X,\BQ) \subset \dots \subset  H^d(X, \BQ)
\end{equation}
the \emph{perverse filtration} associated with the map $f$. The perverse filtration (\ref{balabalah}) is called \emph{multiplicative} if
\[
P^f_kH^\ast(X, \BQ) \cup P^f_{k'}H^\ast(X, \BQ) \subset P^f_{k+k'}H^\ast(X, \BQ)
\] 
for the cup product $\cup$ and any $k, k' \geq 0$.

We refer to \cite{BBD, dCM0, dCM1} for more details about derived categories of constructible sheaves and perverse $t-$structures. See also Section 1 for discussions on perverse filtrations.

\subsection{Hilbert schemes of points on elliptic surfaces} We first locate tautological cohomology classes for the Hilbert scheme of points on an elliptic surface with respect to the natural perverse filtration.\footnote{The assumption that the surface admits an elliptic fibration is essential. See Remark \ref{rmkkkk}.}

Let $S$ be an irreducible nonsingular projective surface fibered over a nonsingular curve $C$,
\[
\pi: S \rightarrow C,
\]
such that a general fiber of $\pi$ is an elliptic curve. We assume that
\[
\pi^{[n]}: S^{[n]} \rightarrow C^{(n)}.
\]
is the induced morphism between the Hilbert scheme $S^{[n]}$ of $n$ points on $S$ and the symmetric powers $C^{(n)}$. Let $\bar{\pi}^{[n]}$ be the Cartesian product 
\[
\bar{\pi}^{[n]}= \pi^{[n]}\times \pi: S^{[n]} \times S \rightarrow C^{(n)}\times C.
\]
\begin{comment}
In Section 1, we define splitting of the perverse filtrations associated with $\pi^{[n]}$ and $\bar{\pi}^{[n]}$ on the cohomology groups $H^\ast(S^{[n]}, \BQ)$ and $H^\ast(S^{[n]}\times S, \BQ)$. More precisely, we construct decompositions
\[
H^\ast(S^{[n]}, \BQ) = \bigoplus_{i\geq 0} G_iH^\ast (S^{[n]}, \BQ), \quad H^\ast(S^{[n]} \times S, \BQ) = \bigoplus_{i\geq 0} G_iH^\ast (S^{[n]} \times S, \BQ)
\]
satisfying
\[
\begin{split}
P^{\pi^{[n]}}_k H^\ast(S^{[n]}, \BQ) & = \bigoplus_{k\leq 0} G_iH^\ast (S^{[n]}, \BQ), \\
P^{\bar{\pi}^{[n]}}_k H^\ast(S^{[n]}\times S, \BQ) & = \bigoplus_{k\leq 0} G_iH^\ast (S^{[n]} \times S, \BQ).
\end{split}
\]
The perverse filtration assocaited to $\pi^{[n]}$ is called \emph{strongly multiplicative} with respect to the splitting $G_\bullet H^\ast (S^{[n]}, \BQ)$ if
\[
G_i H^\ast(S^{[n]}, \BQ) \cup G_j H^\ast (S^{[n]}, \BQ) \subset G_{i+j}H^\ast(S^{[n]}, \BQ).
\]
\end{comment}

\begin{thm}\label{thm1}
Assume the perverse filtration associated with $\pi^{[n]}$ is multiplicative for any $n \geq 0$. We have
\begin{equation}\label{univ_class}
c_{k}(\CO_{Z_n}) \in P^{\bar{\pi}^{[n]}}_{k}H^{2k}(S^{[n]}\times S, \BQ)
\end{equation}
where $Z_n \subset S^{[n]}\times S$ is the universal subscheme. 
\end{thm}

\begin{comment}
\item[(ii)] Assume the perverse filtration associated with $\pi^{[n]}$ is strongly multiplicative for any $n\geq 1$. If the condition 
\begin{equation}\label{thm0.1.2}
c_k(\CO_{Z_n}) \in G_{k}H^{2k}(S^{[n]}\times S, \BQ)
\end{equation}
holds for $n=1$, then (\ref{thm0.1.2}) also holds for any $n\geq 1$.
\end{comment}

It was shown in \cite{LQW} that the cohomology ring $H^\ast(S^{[n]}, \BQ)$ is generated by the tautological classes --- the K\"unneth factors of $c_k(\CO_{Z_n})$ in $H^\ast(S^{[n]}, \BQ)$. Hence Theorem \ref{thm1} provides a complete description of the perverse filtration associated with $\pi^{[n]}$ via the tautological classes. 

\begin{cor}\label{cor1}
Let $\pi: S\rightarrow \BP^1$ be an elliptic $K3$ surface, then (\ref{univ_class}) holds.
\end{cor}

\subsection{Moduli of parabolic Higgs bundles}
Our main motivation for the study of perverse filtrations for Hibert schemes is the $P=W$ conjecture \cite{dCHM1} for (parabolic) Higgs bundles; see Section 0.4. 

In \cite{Gr}, Gr\"ochenig described five infinite families of moduli spaces of parabolic Higgs bundles labeled by the affine Dynkin diagrams $\tilde{A}_0$, $\tilde{D}_4$, $\tilde{E}_6$, $\tilde{E}_7$, and $\tilde{E}_8$, which extends a result of Gorsky--Nekrasov--Rubtsov \cite{GNR}. For each Dynkin diagram above, there attached a sequence of moduli spaces
\[
\{ M_n \}_{n \geq 1}, \quad \mathrm{dim} M_n = 2n,
\]
which are associated with an elliptic curve $E$ with an action of a finite group $\Gamma$. The variety $M_n$ can be realized as either the moduli space of rank $n|\Gamma|$ stable Higgs bundles on the orbifold curve \footnote{Here $|\Gamma|$ denotes the size of the finite group $\Gamma$.}
\[
\BP_\Gamma = [E / \Gamma],
\]
or the moduli space of certain stable parabolic Higgs buundles on the coarse moduli space
\[
|\BP_\Gamma|=
\begin{cases}
 E, & \tilde{A}_0~~~~\textup{case};\\
 \BP^1, & \textup{the others}
\end{cases}
\]
of the same rank; see \cite{Gr} for more details. Gr\"ochenig further showed that $M_1$ is a nonsingular surface elliptically fibered over the affine line,
\[
\pi_1: M_1 \rightarrow \BC,
\]
and $M_n$ is the Hilbert scheme of $n$-points on $M_1$ with the Hitchin fibration
\[
\pi_n : M_n=M_1^{[n]} \xrightarrow{\pi_1^{[n]}} \BC^n.
\]
In Section 3, we introduce a canonical decomposition 
\begin{equation}\label{perverse_decomp}
H^\ast(M_n, \BQ) = \bigoplus_{i\geq 0}G_i H^\ast(M_n, \BQ)
\end{equation}
splitting the perverse filtration on $H^\ast(M_n, \BQ)$ associated with $\pi_n$, \emph{i.e.},
\[P^{\pi_n}_k H^\ast(M_n, \BQ) = \bigoplus_{i\leq k}G_i H^\ast(M_n, \BQ).
\]
We call such a splitting the perverse decomposition. 

For an orbifold (or a Deligne--Mumford stack) $\CX$, we define $\CI\CX$ to be the corresponding inertia stack equiped with the canonical morphism $\CI\CX \to \CX$. We use $H_{\mathrm{orb}}^\ast(\BP_\Gamma, \BQ)$ to denote the cohomology $H^\ast(\CI\BP_\Gamma, \BQ)$. The following theorem concerns the precise locations of the tautological classes on $M_n$ with respect to the perverse decomposition (\ref{perverse_decomp}).

\begin{thm}\label{thm0.3}
Let $(\CE, \sigma)$ be the universal Higgs bundle over $M_n \times \CI\BP_\Gamma$ pulled back from $M_n \times \BP_\Gamma$. Then the following statements hold:
\begin{enumerate}
    \item[(i)] The tautological classes 
    \[
    \int_{\alpha} c_k(\CE) \in H^\ast(M_n, \BQ), \quad \alpha \in H_{\mathrm{orb}}^\ast(\BP_\Gamma, \BQ)
    \]
    generate the cohomology ring $H^\ast(M_n, \BQ)$.
    \item[(ii)] We have
    \[
    \int_{\alpha} c_k(\CE) \in  G_kH^\ast(M_n, \BQ) \subset P_k^{\pi_n} H^\ast(M_n, \BQ)
    \]
    for any $\alpha \in H_{\mathrm{orb}}^\ast(\BP_\Gamma, \BQ)$.
\end{enumerate}
\end{thm}

\subsection{The P=W Conjecture} Simpson established in \cite{Simp} the non-abelian Hodge theorem for a curve of genus $\geq 2$ and the reductive group $\mathrm{GL}_n$, which gives a canonical diffeomorphism between the moduli space $\CM_{\mathrm{Dol}}$ of rank $n$ stable Higgs bundle and the corresponding character variety $\CM_B$ of rank $n$ stable local systems. 

A striking phenomenon was discovered by de Cataldo, Hausel, and Migliorini in \cite{dCHM1}, that the canonical isomorphism 
\[
H^\ast(\CM_B, \BQ) =  H^\ast(\CM_{\mathrm{Dol}}, \BQ)
\]
induced by the Simpson correspondence is \emph{expected} to identify the weight filtration $W_{2\bullet}H^\ast(\CM_B, \BQ)$ and the perverse filtration $P_\bullet H^\ast(\CM_{\mathrm{Dol}}, \BQ)$ associated with the Hitchin fibration, \emph{i.e.}
\[
W_{2k}H^\ast(\CM_{B}, \BQ)= W_{2k+1}H^\ast(\CM_{B}, \BQ) = P_kH^\ast(\CM_{\mathrm{Dol}}, \BQ), \quad k \geq 0.
\]
Such a phenomenon is refered to as ``the $P=W$ conjecture''. The original $P=W$ conjecture was verified for $n=2$ in \cite{dCHM1}, while the $n \geq 3$ cases are still open.

The main purpose of this paper is to provide a proof of the $P=W$ conjecture for the moduli spaces $M_n$ of parabolic Higgs bundles associated with the affine Dykin diagrams $\tilde{A}_0$, $\tilde{D}_4$, $\tilde{E}_6$, $\tilde{E}_7$, and $\tilde{E}_8$, as described in Section 0.3.

Let $\CM'_n$ be the character variety of stable parabolic local systems associated with $M_n$ via the correspondence \cite{Simp90}. The parabolic non-abelian Hodge theorem was proven in \cite{B1, B2}, which induces a canonical isomorphism of the cohomology groups
\begin{equation}\label{non_abel}
    H^\ast(\CM'_n, \BQ) = H^\ast(M_n, \BQ).
\end{equation}

\begin{thm}\label{P=W}
The $P=W$ conjecture holds for the five families of moduli spaces of parabolic Higgs bundles/local systems labeled by the Dykin diagrams $\tilde{A}_0$, $\tilde{D}_4$, $\tilde{E}_6$, $\tilde{E}_7$, and $\tilde{E}_8$, \emph{i.e.}, under the identification (\ref{non_abel}) we have
\[
W_{2k}H^\ast(\CM'_{n}, \BQ)= W_{2k+1}H^\ast(\CM'_{n}, \BQ) = P_kH^\ast(M_n, \BQ), \quad k \geq 0.
\]
\end{thm}

In fact, we prove a refinement of Theorem \ref{P=W} in Section 3. For every character variety $\CM'_n$ above, we consider the following sub-vector space of $H^d(\CM'_n, \BQ)$:
\[
^k\mathrm{Hdg}^d(\CM'_n) =  W_{2k}H^d(\CM'_n, \BQ) \cap F^k H^d(\CM'_n, \BC) \cap \bar{F}^k H^d(\CM'_n, \BC).
\]
Using a similar argument as in \cite{Shende}, we show in Section 3.5 that the cohomology of the character variety $\CM'_n$ admits a canonical decomposition
\begin{equation}\label{weight_decomp}
H^\ast(\CM'_n, \BQ) = \bigoplus_{k,d}  {^k\mathrm{Hdg}^d(\CM'_n)}.
\end{equation}
The decomposition (\ref{weight_decomp}) is then shown to match the perverse decomposition (\ref{perverse_decomp}), which implies Theorem \ref{P=W} as a conclusion.

Our results provide examples of moduli of parabolic Higgs bundles where the $P=W$ conjecture holds for arbitrary rank. The construction of the perverse decomposition (\ref{perverse_decomp}) and Theorem \ref{thm0.3} play an essential role in the proof.

\subsection{Relation to other work}
An interesting phenomenon has been found in \cite{dCHM3} and extended in \cite[Section 4.4]{Z}, that the perverse filtration for $M_n = M_1^{[n]}$ matches the weight filtration for the Hilbert scheme ${\CM_1'}^{[n]}$. However, this exchange of filtrations is \emph{not} $``P=W"$, since a character variety is an affine variety which cannot be realized as the Hilbert scheme of points.

\subsection{Conventions} Throughout the paper, we work over the complex numbers $\BC$. For a projective variety $X$, a quasi-projective variety $Y$, and an object $\CF \in D^b(X\times Y)$, we say that a correspondence $g: H^\ast(X, \BQ) \to H^\ast(Y, \BQ)$ is induced by $\CF$ if $g$ is given by
\[
g(\alpha) = \mathrm{pr}_{Y\ast}\left(\mathrm{pr}^\ast_X (\alpha \cup \mathrm{td}_X)\cup \mathrm{ch}(\CF) \right).
\]
For $\gamma \in H^\ast(X, \BQ)$ and $\alpha \in H^\ast(X\times Y, \BQ)$, we use $\int_{\gamma} \alpha$ to denote the class
\[
\mathrm{pr}_{Y\ast}\left(\mathrm{pr}^\ast_X \gamma \cup \alpha \right) \in H^\ast(Y, \BQ).
\]

\subsection{Acknowledgement}
We are grateful to Tim-Henrik B\"ulles, Mark de Cataldo, Davesh Maulik, Rahul Pandharipande, and Qizheng Yin for discussions about Hilbert schemes, perverse filtrations, and Higgs bundles, and to Olivier Biquard for kindly pointing out references.

\section{Perverse decompositions for Hilbert schemes}

\subsection{Perverse filtrations} We define the perverse filtration associated with any proper morphism $f: X \rightarrow Y$ between nonsingular varieties.

Let 
\[
r(f)=\dim X\times_YX-\dim X
\]
be the defect of semismallness. In particular, we have $r(f) =r$ when $f$ has relative dimension $r$, and we have $r(f)=0$ if $f$ is semismall.

Similar to Section 0.1, we define the perverse filtration $P_k^f H^\ast(X, \BQ)$ by
\begin{multline*}
    P^f_kH^d(X,\BQ)   \\
     :=\mathrm{Im}\{{H}^{d-\dim X+r(f)}\Big{(}X,~~~~~ {^\mathfrak{p}\tau_{\le k}}(Rf_*\BQ_X[\dim X-r(f)])\Big{)} \to H^d(X,\BQ)\}.
\end{multline*}

Once we choose a decomposition 
\begin{equation}\label{decomp}
Rf_*\BQ_X[\dim X-r(f)]\cong \bigoplus_{i=0}^{2r(f)}\mathcal{P}_i[-i] \in D^b_c(Y)
\end{equation}
with $\mathcal{P}_i$ perverse sheaves on $Y$ \cite{BBD, dCM1}, then the perverse filtration associated with the morphism $f$ can be computed as
\[
\displaystyle P_k^fH^*(X,\BQ)=\mathrm{Im}\{{H}^{*-\dim X+r(f)}(\oplus_{i=0}^k\mathcal{P}_i[-i])\to H^*(X,\BQ)\}.
\]
In general, a perverse filtration does not allow a natural splitting, since the decomposition (\ref{decomp}) is dependent on the choice of an isomorphism which is not canonical.

We define the \emph{perversity} $\mathfrak{p}^f(\alpha)$ of a class $\alpha \in H^d(X, \BQ)$ to be the integer $k$ satisfying $\alpha \in P_{k}^fH^d(X, \BQ)$ and $\alpha \not\in P^f_{k-1}H^d(X, \BQ)$. Since the perverse filtration is concentrated in the degrees $[0, 2r(f)]$\footnote{We call that a filtration on a vector space is concentrated in the degrees $[a, b]$, if the $k$-th graded piece of the filtration is empty when $k \not\in [a,b]$.}, we have 
\[
0 \leq \mathfrak{p}^f(\alpha) \leq 2r(f)
\]
for any class $\alpha \in H^\ast(X, \BQ)$. 

Let $f_1: X_1 \rightarrow Y_1$ and $f_2: X_2 \rightarrow Y_2$ be proper morphisms between nonsingular quasi-projective varieties. We recall the following proposition from \cite{Z} concerning the perverse filtration associated with the product
\[
f_1\times f_2: X_1 \times X_2 \rightarrow Y_1 \times Y_2.
\]

\begin{prop}{\cite[Proposition 2.1]{Z}}\label{kunneth}
For $\alpha_1 \in H^\ast(X_1, \BQ)$ and $\alpha_2 \in H^\ast(X_2, \BQ)$, we have
\[
\mathfrak{p}^{f_1 \times f_2}(\alpha_1 \boxtimes \alpha_2) = \mathfrak{p}^{f_1}(\alpha_1) + \mathfrak{p}^{f_2}(\alpha_2).
\]
\end{prop}

\subsection{Perverse decompositions}
In this section, we introduce perverse decompositions for symmetric powers and Hilbert schemes of points associated with a surface, which split the perverse filtrations on the corresponding cohomology groups.

Let $\pi: S \rightarrow C$ be a proper surjective morphism from a nonsingular quasi-projective surface to a nonsingular curve. Then $r(\pi)=1$, and $H^\ast(S, \BQ)$ admits a perverse filtration of length 2,
\begin{equation}\label{perv123}
P^{\pi}_0 H^\ast(S, \BQ) \subset P^{\pi}_1 H^\ast(S, \BQ)  \subset P^{\pi}_2 H^\ast(S, \BQ) = H^\ast(S, \BQ).
\end{equation}
Throughout Section 2, we \emph{fix} a decomposition 
\begin{equation}\label{fixx}
H^\ast(S, \BQ) = \bigoplus_{i\geq 0} G_i H^\ast(S, \BQ)
\end{equation}
splitting the filtration (\ref{perv123}). More precisely, we have 
\[
P^{\pi}_kH^\ast(S, \BQ)= \bigoplus_{i\leq k}G_iH^\ast(S, \BQ)
\]
for $k\in \{0,1,2\}$.

We define the decomposition of $H^\ast(S^{n}, \BQ)$ to be
\[
G_kH^\ast(S^n, \BQ)=\left\langle \alpha_1\boxtimes \dots \boxtimes \alpha_n; ~~~ \alpha_i \in G_{k_i}H^\ast(S, \BQ),~~~ \sum_{i}k_i=k \right\rangle.
\]
Due to Proposition \ref{kunneth}, this decomposition splits the perverse filtration of $\pi^n : S^n \rightarrow C^n$. Taking the $\mathfrak{S}_n$-invariant part yields the perverse decomposition for the symmetric power $S^{(n)}$ associated with $\pi^{(n)}: S^{(n)} \rightarrow C^{(n)}$,
\[
H^\ast(S^{(n)}, \BQ) = \bigoplus_{i\geq 0} G_iH^\ast (S^{(n)}, \BQ).
\]
The perverse decomposition for the product $S^{(n_1)}\times S^{(n_2)} \times \dots \times S^{(n_k)}$ is then defined by a similar way using the K\"unneth decomposition.

We consider the Hilbert scheme $S^{[n]}$ of $n$ points on $S$, which is a $2n$-dimensional nonsingular variety parametrizing length $n$ 0-dimensional subschemes in $S$. Rich structures of the cohomology of the Hilbert scheme of points have been found and studied intensively in the last decades; see \cite{Go1, Go2, Na, L, LS, LQW}.
Now we assume
\[\pi^{[n]}: S^{[n]} \rightarrow C^{(n)}\]
is the natural morphism associated with $\pi$. For a partition 
\[
\nu = 1^{a_1}2^{a_2}\cdots n^{a_n}
\]
of $n$, we use $S^{(\nu)}$ to denote the variety $S^{(a_1)} \times S^{(a_2)} \times \cdots \times S^{(a_n)}$. By \cite{Go2}, the cohomology group $H^d(S^{[n]}, \BQ)$ admits a \emph{canonical} decomposition
\begin{equation}\label{GS_decomp}
H^d(S^{[n]}, \BQ) = \bigoplus_{\nu} H^{d+2l(\nu)-2n}(S^{(\nu)}, \BQ)
\end{equation}
where $\nu$ runs through all partitions of $n$ and $l(\nu)$ is the length of $\nu$. We define \[
G_kH^d(S^{[n]}, \BQ) = \bigoplus_{\nu} G_{k+l(\nu)-n}H^{d+2l(\nu)-2n}(S^{(\nu)}, \BQ)  
\]
to be the sub-vector space of $H^d(S^{[n]}, \BQ)$ under the identification (\ref{GS_decomp}). The following proposition is obtained from \cite[Proposition 4.12]{Z}.

\begin{prop}
The decomposition
\[
H^\ast(S^{[n]}, \BQ) = \bigoplus_{i\geq 0} G_iH^\ast (S^{[n]}, \BQ)
\]
defined above splits the perverse filtration associated with $\pi^{[n]}: S^{[n]} \rightarrow C^{(n)}$.
\end{prop}

We have constructed perverse decompositions on the cohomology groups $H^\ast(S^{[n]}, \BQ)$. Again, by Proposition \ref{kunneth}, we obtain perverse decompositions for the product of Hilbert schemes
\[
S^{[n_1]}\times S^{[n_2]} \times \dots S^{[n_k]}.
\]

\subsection{Nested Hilbert schemes}
We consider the nested Hilbert scheme
\[
S^{[n,n+1]}=\{(\xi,\eta): \xi\in S^{[n+1]}, \eta\in S^{[n]}, \xi\subset \eta\} \subset S^{[n+1]}\times S^{[n]}
\]
which is a nonsingular and irreducible variety of dimsneion $2n+2$ \cite{Ch}. In this section, we extend perverse decompositions to the nested Hilbert schemes associated with $\pi: S \rightarrow C$, which plays a crucial role in the proofs of Theorems \ref{thm1} and \ref{P=W}.

The variety $S^{[n,n+1]}$ admits the following natural morphisms
\begin{equation}\label{incidence123}
\begin{tikzcd}
S& S^{[n,n+1]}\arrow[l, "\rho"] \arrow[d,"p_n"]\arrow[r,"p_{n+1} "]& S^{[n+1]} \\
& S^{[n]} &
\end{tikzcd}
\end{equation}
where $p_{n}$ and $p_{n+1}$ are the two projections, and $\rho$ sends $(\xi,\eta)$ to the point $\{\eta \smallsetminus \xi\} \in S$. The morphism
\[
q_n = (p_n, \rho): S^{[n, n+1]} \rightarrow  S^{[n]}\times S
\]
is realized as the blow-up of the codimension 2 incidence $Z_n \subset S^{[n]} \times S$ with the exceptional divisor
\begin{equation}\label{exceptional123}
E_{n+1} = \{(\xi, \eta) \in S^{[n+1,n]}: \mathrm{Supp}(\xi) = \mathrm{Supp}(\eta)\} \subset S^{[n, n+1]}.
\end{equation}

Let $g_n: S^{[n, n+1]} \rightarrow S^{(n)} \times S$ be the composition
\[
S^{[n,n+1]} \xrightarrow{q_n} S^{[n]}\times S \rightarrow S^{(n)}\times S,
\]
and let 
\[
\bar{g}_n = (\pi^{(n)} \times \pi)\circ g_n: S^{[n,n+1]} \to C^{(n)} \times C.
\]

Following \cite{Ch}, we define 
\[
S^{(\nu, j)}=
\begin{cases}
S^{(\nu)}\times S & j=0;\\
S^{(\nu^\flat)}\times S& j>0 \textrm{ and } a_j>0;\\
\phi & j>0 \textrm{ and } a_j=0,
\end{cases}
\]
where $\nu=1^{a_1}\cdots n^{a_n}$ is a partition of $n$ and $\nu^\flat$ is defined as a partition of $n-j$ obtained from $\nu$ by reducing $a_j$ by $1$. This definition matches the one in \cite[Section 3.3]{dCM4}. In particular $S^{(\nu,j)}$ is a Cartesian product of symmetric powers of $S$. Hence the cohomology of $S^{(\nu,j)}$ carries a natural perverse decomposition as defined in Section 1.2, 
\[
H^\ast(S^{(\nu,j)}, \BQ)=\bigoplus_{k\ge0} G_k H^\ast(S^{(\nu,j)},\BQ).
\]
By \cite[Theorem 3.3.1]{dCM4}, there is a \emph{canonical} isomorphism
\begin{equation}\label{decomp1234}
H^d(S^{[n,n+1]}, \BQ)=\bigoplus_{\nu,j} H^{d-2m(\nu,j) }(S^{(\nu,j)},\BQ)
\end{equation}
where
\[
m(\nu,j) = 
\begin{cases}
n-l(\nu), & j=0;\\
n+1-l(\nu), & j>0.
\end{cases}
\]
We define the following decomposition of $H^d(S^{[n,n+1]},\BQ)$ under the identification (\ref{decomp1234}), 
\[
G_k H^d(S^{[n,n+1]},\BQ)=\bigoplus_{\nu,j} G_{k-m(\nu,j)}H^{d-2m(\nu,j)}(S^{(\nu,j)},\BQ).
\]

Since $S^{(\nu,j)}$ is a product of symmetric powers of $S$ for every $(\nu,j)$, the method in \cite[Proposition 4.12]{Z} together with the decomposition theorem \cite[Theorem 3.3.1]{dCM4} implies the following proposition.

\begin{prop}
The decomposition 
\[
H^\ast(S^{[n,n+1]},\BQ)=\bigoplus_{i\ge 0}G_i H^\ast(S^{[n,n+1]},\BQ)
\]
defined above splits the perverse filtration associated with \[
\bar{g}_n:S^{[n,n+1]}\rightarrow C^{(n)}\times C.\]
\end{prop}

\subsection{Functoriality}
The natural morphisms \[
q_n: S^{[n,n+1]} \to S^{[n]}\times S
\]
and 
\[
p_{n+1}: S^{[n,n+1]} \rightarrow S^{[n+1]}
\]
are generically finite. In the following, we prove some functoriality results concerning the morphisms $q_n$ and $p_{n+1}$ with respect to the perverse decompositions constructed in Sections 1.2 and 1.3.

\begin{prop}\label{prop1.4}
Let $q_n^\ast: H^\ast(S^{[n]}\times S, \BQ) \to H^\ast(S^{[n,n+1]}, \BQ)$ be the pullback morphism. We have
\[
q_n^* \left(G_kH^\ast(S^{[n]}\times S, \BQ) \right) \subset G_kH^\ast(S^{[n,n+1]},\BQ).
\]
\end{prop}

\begin{proof}
Consider the following Cartesian diagrams,
\[
\begin{tikzcd}
 \tilde{\Gamma}\arrow[d]\arrow[r] & S^{[n,n+1]}\arrow[d, "q_n"]\\
 \Gamma \arrow[d]\arrow[r] & S^{[n]}\times S \arrow[d]\\
S^{(\nu,j)}\arrow[r] & S^{(n)}\times S
\end{tikzcd}
\]
where $S^{(\nu,j)}\to  S^{(n)}\times S$ is the natural morphism to the corresponding stratum \cite[Section 3]{dCM4}. 

On one hand, by \cite[Theorem 5.4.1]{dCM3} and the K\"unneth decomposition, there is a canonical decomposition
\[
H^\ast(S^{[n]}\times S,\BQ)=\bigoplus_\nu H^\ast(S^{(\nu)}\times S,\BQ),
\]
where the direct summands are canonically identified with the image of the correspondences \[
\Gamma_\ast:H^\ast(S^{(\nu)}\times S,\BQ)\rightarrow H^\ast(S^{[n]}\times S,\BQ).
\]
On the other hand, by \cite[Theorem 3.3.1]{dCM4} there is a canonical decomposition
\[
H^\ast(S^{[n,n+1]},\BQ)=\bigoplus_{\nu,j} H^\ast(S^{(\nu,j)},\BQ),
\]
where the direct summands are identified with the images of the correspondences $\tilde{\Gamma}_\ast:H^\ast(S^{(\nu,j)},\BQ)\rightarrow H^\ast(S^{[n,n+1]},\BQ)$. Since $S^{(\nu)}\times S=S^{(\nu,0)}$, we obtain a commutative diagram
\[
\begin{tikzcd}
H^\ast(S^{[n]}\times S,\BQ)\arrow[r,"q_n^\ast"]& H^\ast(S^{[n,n+1]},\BQ)\\
H^\ast(S^{(\nu)}\times S,\BQ)\arrow[r,equal] \arrow[u,hook]& H^\ast(S^{(\nu,0)},\BQ)\arrow[u,hook]
\end{tikzcd}
\]
for each partition $\nu$ of $n$. In particular $q_n^\ast$ preserves the perverse decompositions.
\end{proof}

\begin{prop}\label{prop1.5}
Let
\[
p_{n+1\ast}: H^\ast(S^{[n,n+1]}, \BQ)\rightarrow H^\ast(S^{[n+1]}, \BQ)
\]
be the Gysin pushforward. Then we have
\[
p_{n+1\ast} \left(G_kH^\ast(S^{[n,n+1]},\BQ) \right)\subset G_kH^\ast(S^{[n+1]},\BQ).
\]
\end{prop}

\begin{proof}
Consider the commutative diagram
\[
\begin{tikzcd}
S^{[n,n+1]}\arrow[rr,"p_{n+1}"]\ar[rd,"f"]\ar[d,"g_n"]& & S^{[n+1]}\arrow[ld,"h"]\\
S^{(n)}\times S \arrow[r,"q"]&S^{(n+1)} & \\
S^{(\nu,j)}\arrow[u,hook,"\iota"]\arrow[r,"q_{\nu,j}"]& S^{(\nu')}\arrow[u,hook,"\iota'"],
 \end{tikzcd}
\]
where $\nu=1^{a_1}\cdots n^{a_n}$ is a partition of $n$, and 
\[
\nu'=
\begin{cases}
1^{a_1+1}2^{a_2}\cdots n^{a_n}, & j=0;\\
1^{a_1}\cdots j^{a_j-1}(j+1)^{a_{j+1}+1}\cdots n^{a_n}, &j>0.
\end{cases}
\]
By \cite[Theorem 3.3.1]{dCM4}, we have 
\[
R{g_n}_\ast \BQ_{S^{[n,n+1]}}=\bigoplus_{\nu,j}\iota_\ast \BQ_{S^{(\nu,j)}}[-2m(\nu,j)],
\]
Since $q\circ \iota:S^{(\nu,j)}\rightarrow S^{(n+1)}$ is a finite map, the object $q_\ast \iota_\ast \BQ_{S^{(\nu,j)}}[m(\nu,j)]$ is a perverse sheaf supported on the locus
\[
\iota'(S^{(\nu')}) \subset S^{(n+1)},
\]
where $\iota': S^{(\nu')} \to S^{(n+1)}$ is the natural map to the corresponding stratum as in \cite{Go2}. We compare the following two decompositions of perverse sheaves on $S^{(n+1)}$,
\begin{align*}
   Rf_\ast \BQ_{S^{[n,n+1]}} &= \bigoplus_{\nu,j}q_\ast\iota_\ast \BQ_{S^{(\nu,j)}}[-2m(\nu,j)], \\
 Rh_\ast \BQ_{S^{[n+1]}} &= \bigoplus_{\nu'}{\iota'}_\ast\BQ_{S^{(\nu')}}[2l(\nu')-2n].
\end{align*}
Since perverse sheaves on different irreducible supports do not have morphisms between them, we obtain that the pushforward morphism 
\begin{equation}\label{123333}
Rf_\ast \BQ_{S^{[n,n+1]}}\rightarrow Rh_*\BQ_{S^{[n+1]}}
\end{equation}
splits canonically to
\[
q_\ast\iota_\ast \BQ_{S^{(\nu,j)}}\rightarrow \iota'_\ast\BQ_{S^{(\nu')}}
\]
for every partition $\nu$ and $j \geq 0$.
Since the pushforward
\[
p_{n+1 \ast}: H^\ast(S^{[n,n+1]}, \BQ) \to H^\ast(S^{[n+1]}, \BQ)
\]
is induced by the morphism (\ref{123333}) of sheaves, we conclude that $p_{n+1 \ast}$ is given by 
\[
(q_{\nu,j})_\ast: H^\ast(S^{(\nu,j)},\BQ)\rightarrow H^\ast(S^{(\nu')},\BQ)
\]
under the decompositions (\ref{GS_decomp}) and (\ref{decomp1234}) for every partition $\nu$ and $j \geq 0$.  

Finally, since the morphism $q_{\nu,j}$ is finite, the induced pushforward $(q_{\nu,j})_\ast$ preserves the perverse decomposition. We complete the proof.
\end{proof}

Although the perverse decompositions depend on a choice of (\ref{fixx}), the perverse filtrations are canonical. The following corollary concerns the functoriality of perverse filtrations.

\begin{cor}\label{functoriality}
We have
\begin{align*}
     q_n^\ast \left( P^{\bar{\pi}^{[n]}}_k H^\ast(S^{[n]}\times S, \BQ) \right) & \subset  P^{\bar{g}_n}_k H^\ast(S^{[n,n+1]}, \BQ),\\
    p_{n+1 \ast}\left( P^{\bar{g}_n}_k H^\ast(S^{[n,n+1]}, \BQ) \right)  &\subset P_k^{\pi^{[n+1]}} H^\ast(S^{[n+1]}, \BQ).
\end{align*}
\end{cor}

In general, for a commutative diagram 
\begin{equation}\label{diagram123333}
\begin{tikzcd}
X\ar[rr,"f"]\ar[rd,"h"]& & Y\ar[ld,"g"]\\
 &B& 
\end{tikzcd}
\end{equation}  
with all varieties nonsingular and all morphisms proper, the perverse filtration 
\[
P^g_kH^\ast(Y, \BQ) \quad (\textrm{resp.} ~~~~~~~~ P^h_kH^\ast(X, \BQ))
\]
may not be preserved by the pullback $f^\ast$ (resp. the pushforward $f_\ast$); see Example \ref{Example} below. Hence Corollary \ref{functoriality} relies on the geometry of (nested) Hilbert schemes.

\begin{Example}\label{Example}
\begin{enumerate}
    \item[(i)] In the diagram (\ref{diagram123333}), we let $X=\mathrm{Bl}_{\mathrm{pt}} \BP^3$ and $Y=B=\BP^3$. Assume $f,g,h$ are the natural maps. Then we have $\mathfrak{p}^h([X])=1$ and $\mathfrak{p}^g([Y])=0$; see \cite[Example 1.5]{Z}. Hence
    \[
    f^\ast P^g_0 H^0(Y,\BQ) \not\subset P^h_0H^0(X,\BQ).
    \]
    \item[(ii)] Let $X=B=\mathrm{pt}$ and $Y=\BP^1$. Assume $f,g,h$ are the natural maps.
    We have $\mathfrak{p}^h([\mathrm{pt}])=0$ and $\mathfrak{p}^{g}([\mathrm{pt}])=2$, and in particular, 
    \[f_\ast P_0^hH^\ast(X,\BQ) \not\subset P_0^gH^\ast(Y,\BQ).
    \]
\end{enumerate}
\end{Example}

\subsection{Strong multiplicativity of the perverse decomposition} \label{multiplicativity} In this section, we prove the \emph{strong multiplicativity} with respect to the canonical perverse decomposition defined in Section 1.2 for the the Hilbert scheme of points on a surface with numerically trivial canonical divisor. This is a variant of \cite[Theorem 4.18]{Z} and \cite[Theorem 5.6]{Z}, where only perverse filtrations are considered. Since the proofs are similar, we will give a sketch and focus on pointing out the differences. 

Let $\pi:S\rightarrow C$ be a proper surjective morphism from a nonsingular quasi-projective surface to a nonsingular quasi-projective curve and let 
\begin{equation}\label{decomp00}
H^\ast(S,\BQ)=\bigoplus_i G_iH^\ast(S,\BQ)
\end{equation}
be a fixed perverse decomposition associated with $\pi:S\rightarrow C$. We say that the decomposition (\ref{decomp00}) satisfies the condition $(\dagger)$ if the following two properties are satisfied:
\begin{enumerate}
    \item The perverse decomposition is \emph{strongly} multiplicative, \emph{i.e.}
    \[
    G_iH^\ast(S,\BQ)\cup G_j H^\ast(S,\BQ)\subset G_{i+j}H^\ast(S,\BQ).
    \]
    \item The pushforward morphism along the embedding $\Delta_n:S\rightarrow S^n$ of the small diagonal satisfies
    \[
    \Delta_{n\ast}:G_iH^\ast(S,\BQ)\rightarrow G_{i+2(n-1)}H^\ast(S^n,\BQ).
    \]
\end{enumerate}

\begin{prop}\label{Prop1.8}
Suppose that the perverse decomposition (\ref{decomp00}) associated with $\pi: S \to C$ satisfies the condition ($\dagger$). Suppose further that either 
\begin{enumerate}
    \item[(i)] $S$ is projective and has numerically trivial canonical bundle, or
    \item[(ii)] $S$ has trivial canonical bundle and admits a nonsingular compactification $\bar{S}$, such that the restriction map $H^\ast(\bar{S},\BQ)\rightarrow H^\ast(S,\BQ)$ is surjective.
\end{enumerate}
Then the perverse decomposition on $H^\ast(S^{[n]},\BQ)$ is strongly multiplicative, \emph{i.e.}
\[
G_iH^\ast(S^{[n]},\BQ)\cup G_jH^\ast(S^{[n]},\BQ)\subset G_{i+j}H^\ast(S^{[n]},\BQ).
\]
\end{prop}

\begin{proof}
We closely follow the notation in \cite[Section 4]{Z}. First of all, a chosen basis of $G_iH^\ast(S,\BQ)$ forms a basis of $H^\ast(S,\BQ)$. We use this basis to define the abstract perversity on $H^\ast(S,\BQ)\{\mathfrak{S}_n\}$ as \cite[Definition 4.15]{Z}. Then $H^\ast(S,\BQ)\{\mathfrak{S}_n\}$ is endowed with a canonical perverse decomposition compatible with the one on $H^\ast(S^{[n]},\BQ)$. Furthermore, by \cite[Definitions 4.1 and 4.3]{Z}, the group $H^\ast(S,\BQ)^I$ has a canonical perverse decomposition for any set $I$.

The strong multiplicativity of $H^\ast(S,\BQ)$ implies that the strengthened conclusion of \cite[Lemma 4.19]{Z} holds, that the morphism
\[
\varphi^\ast:H^\ast(S,\BQ)^I\rightarrow H^\ast(S,\BQ)^J
\]
preserves the perverse decompositions, \emph{i.e.}
\[
\varphi^\ast G_iH^\ast(S,\BQ)^I\subset G_iH^\ast(S,\BQ)^J.
\]
By the condition (2) on small diagonal embeddings, the conclusion of \cite[Lemma 4.20]{Z} can be strengthened that the pushforward
\[
\varphi_\ast: H^\ast(S,\BQ)^J\rightarrow H^\ast(S,\BQ)^I
\]
increases the perversity by $2(|I|-|J|)$ in the perverse decompositions, \emph{i.e.}
\[
\varphi_\ast G_iH^\ast(S,\BQ)^J\subset  G_{i+2(n-1)}H^\ast(S,\BQ)^I.
\]
Now the proposition follows from the proofs of \cite[Theorems 4.18 and 5.6]{Z} together with the strengthened versions of Lemmas 4.19 and 4.20 in \cite{Z}.
\end{proof}

\section{Tautological classes on Hilbert schemes}
\subsection{Overview} We assume that $\pi: S \rightarrow C$ is a fibration from a nonsingular irreducible projective surface to a nonsingular curve such that a general fiber is an elliptic curve. We use the tools developed in Section 1 to analyze the perverse filtration of the Hilbert scheme $S^{[n]}$ associated with the morphism
\begin{equation}\label{hilb_mor}
\pi^{[n]}: S^{[n]} \rightarrow C^{(n)},
\end{equation}
and prove Theorem \ref{thm1}.

By \cite{LQW}, the tautological classes
\begin{equation}
\int_{\gamma} c_k(\CO_{Z_n}) \in H^\ast(S^{[n]}, \BQ), \quad \gamma \in H^\ast(S, \BQ)
\end{equation}
are the generators of the ring $H^\ast(S^{[n]}, \BQ)$. Theorem \ref{thm1} together with Proposition \ref{kunneth} calculates the perversity of these generators on the Hilbert scheme $S^{[n]}$. 

Since \cite{Z} shows that the perverse filtration of $H^\ast(S^{[n]}, \BQ)$ associated with (\ref{hilb_mor}) is multiplicative when $S$ has numerically trivial canonical bundle, we deduce Corollary \ref{cor1} from Theorem \ref{thm1}.

\subsection{Exceptional divisors}
In this section, we assume that $\pi: S\to C$ is a proper surjective morphism from a nonsingular quasi-projective surface to a nonsingular curve $C$. 

Let $\partial S^{[n]}$ be the boundary divisor given by the locus of $S^{[n]}$ where the subschemes have length $\leq n-1$. For any choice  (\ref{fixx}), we obtain a perverse decomposition 
\[
H^\ast(S^{[n]}, \BQ) = \bigoplus_{i\geq 0} G_iH^\ast (S^{[n]}, \BQ)
\]
by the discussion in Section 1. The following lemma calculates the perversity of the boundary divisor
\[
\partial S^{[n]} \in H^2(S^{[n]}, \BQ).
\]

\begin{lemma}\label{perv_exc}
For any choice of (\ref{fixx}), We have
\[
\partial S^{[n]}\in G_1H^2(S^{[n]}, \BQ).
\]
In particular $\mathfrak{p}^{\pi^{[n]}}(\partial S^{[n]}) =1$.
\end{lemma}

\begin{proof}
Consider the following Cartesian diagram,
\[
\begin{tikzcd}
\Gamma\arrow[r]\arrow[d] & S^{[n]}\arrow[d]\\
S^{(1^{n-2}2^1)}\arrow[r, "\iota"] & S^{(n)}
\end{tikzcd}
\]
where $\iota=\iota_{(1^{n-2}2^1)}: S^{(1^{n-2}2^1)} \to S^{(n)}$ as in the proof of Proposition \ref{prop1.5}. We have 
\[
[\partial S^{[n]}] =  \Gamma_*([S^{(1^{n-2}2^1)}])
\]
where $\Gamma_\ast:H^\ast(S^{(1^{n-2}2^1)}, \BQ)\rightarrow H^\ast(S^{[n]}, \BQ)$ is viewed as a correspondence. 
Hence $\partial S^{[n]}$ is identified with the fundamental class
\[
1\in H^\ast(S^{(1^{(n-2)}2^1)}, \BQ)[-2]
\]
in the canonical decomposition (\ref{GS_decomp}). We conclude the lemma by the definition of the perverse decomposition for $S^{[n]}$.
\end{proof}

We recall that the nested Hilbert scheme is realized as the blow-up of $Z_n \subset S^{[n]} \times S$ with the exceptional divisor $E_{n+1}$ given by (\ref{exceptional123}). The following relation from \cite[Lemma 3.7]{L} expresses $E_{n+1}$ as boundary divisors of Hilbert schemes:
\begin{equation}\label{exceptional}
    E_{n+1} = \frac{1}{2}\cdot \Big{(}p_{n+1}^\ast \partial S^{[n+1]} - p_n^\ast \partial S^{[n]}\Big{)} \in H^2(S^{[n, n+1]}, \BQ).
\end{equation}
Here we use the notation as in the diagram (\ref{incidence123}). 

For convenience, we always write 
\[
f\times \mathrm{id}: X_1 \times S \to  X_2 \times S
\]
as $\tilde{f}: X_1\times S \to X_2\times S$. The following short exact sequence from \cite[Page 193]{L} compares the universal families of $S^{[n]}$ and $S^{[n+1]}$:
\begin{equation} \label{compare_uni_family}
    0 \to \tilde{\rho}^\ast \CO_{\Delta_S} \otimes \mathrm{pr}^\ast\CO_{S^{[n,n+1]}}(-E_{n+1}) \to {\tilde{p}_{n+1}}^\ast \CO_{Z_{n+1}} \to {\tilde{p}_n}^\ast \CO_{Z_n} \to 0
\end{equation}
where $\mathrm{pr}: S^{[n,n+1]} \times S \to S^{[n, n+1]}$ is the projection.

\subsection{Proof of Theorem \ref{thm1}}
We prove Theorem \ref{thm1} by induction on $n$. The induction base is given by the following lemma.

\begin{lemma}\label{induction_base}
Let $\pi: S \to C$ be a fibration as in Section 2.1. Then
\[
c_k(\CO_{\Delta_S}) \in P^{\pi \times \pi}_kH^{2k}(S\times S , \BQ).
\]
\end{lemma}

\begin{proof}
The cases of $k=0,1,4$ are obvious. By \cite[Proposition 4.17]{Z}, the perverse filtration associated with $\pi$ is multiplicative. Hence we can apply \cite[Proposition 3.8]{Z} and obtain that
\[
\mathfrak{p}^{\pi\times \pi}{[\Delta]} = \mathfrak{p}^{\pi\times \pi}(\Delta_\ast [S])  \leq 2.
\]
This concludes the $k=2$ case. It remains to show that 
\begin{equation}\label{eqn123}
c_3(\CO_{\Delta_S}) \in P^{\pi \times \pi}_3H^6(S\times S , \BQ).
\end{equation}
By the Grothendieck--Riemann--Roch formula, the class $c_3(\CO_{\Delta_S})$ is proportional to $\Delta_\ast c_1(S)$. Since a general fiber of $\pi: S \rightarrow C$ is an elliptic curve, the class $c_1(S)$ is supported on fibers of $\pi$, and therefore
\[
\mathfrak{p}^\pi(c_1(S)) \leq 1.
\]
Hence (\ref{eqn123}) is again deduced from \cite[Proposition 3.8]{Z}.
\end{proof}

\begin{rmk}\label{rmkkkk}
The assumption that a general fiber of $\pi: S\to C$ is an elliptic curve is essential for Theorem \ref{thm1} to hold. For example, if we take $S = \BP^1\times \BP^1$, $C=\BP^1$, and $\pi: \BP^1\times \BP^1 \rightarrow \BP^1$ the natural projection, then the perversity of $c_3(\CO_{\Delta_S})$ is 4. Hence Theorem \ref{thm1} breaks down even when $n=1$.
\end{rmk}

Since we suppose that the perverse filtration on $S^{[n]}$ associated with $\pi^{[n]}$ is multiplicative, the K\"unneth decomposition and Proposition \ref{kunneth} imply that the perverse filtration on $S^{[n]}\times S$ associated with 
\[
\bar{\pi}^{[n]}: S^{[n]} \times S \rightarrow C^{(n)}\times C
\]
is also multiplicative. Hence Theorem \ref{thm1} is equivalent to 
\begin{equation*}
    \mathrm{ch}_k(\CO_{Z_{n}}) \in P^{\bar{\pi}^{[n]}}_{k}H^{2k}(S^{[n]}\times S, \BQ).
\end{equation*}
 
 Now we assume that Theorem \ref{thm1} holds for $S^{[n]}$, and we need to prove it for $S^{[n+1]}$. The short exact sequence (\ref{compare_uni_family}) implies that the class 
 \[
 \tilde{p}_{n+1}^\ast \mathrm{ch}_k(\CO_{Z_{n+1}}) \in H^{2k}(S^{[n,n+1]}, \BQ)
 \]
 can be expressed as
 \begin{multline*}
      \tilde{q}_n^\ast \mathrm{ch}_k(\CO_{Z_n}\boxtimes  \CO_S)  + 
      \\  \sum_{k_1+k_2=k} \tilde{q}_n^\ast\mathrm{ch}_{k_1}(\CO_{S^{[n]}}\boxtimes \CO_{\Delta_S})\cup \tilde{q}_{n+1}^\ast \mathrm{ch}_{k_2}(\CO_{S^{[n,n+1]}}(-E_{n+1})\boxtimes \CO_S).
 \end{multline*}

In the following, we analyze every term above with respect to the perverse filtrations. First we have
\begin{equation}\label{ch1}
\mathrm{ch}_k(\CO_{Z_n}\boxtimes  \CO_S) \in P^{\bar{\pi}^{[n]}}_k  H^{2k} (S^{[n]}\times S ,\BQ)
\end{equation}
due to the induction hypothesis and Proposition \ref{kunneth}. Also, Lemma \ref{induction_base} yields
\begin{equation}\label{ch2}
\mathrm{ch}_{k_1}(\CO_{S^{[n]}}\boxtimes \CO_{\Delta_S}) \in P^{\bar{\pi}^{[n]}}_{k_1}  H^{2k_1} (S^{[n]}\times S ,\BQ).
\end{equation}
By the equation (\ref{exceptional}), the class of the exceptional divisor 
\[
E_{n+1} \in H^2(S^{[n,n+1]}, \BQ)
\]
can be written as
\begin{equation} \label{exc3}
{p^\ast_{n+1}} \partial S^{[n+1]} - {q^\ast_n}( \partial S^{[n]} \times [S])
\end{equation}
Hence we obtain from (\ref{ch1}), (\ref{ch2}), (\ref{exc3}), and Lemma \ref{perv_exc} that
\begin{equation}\label{eqn2018}
\tilde{p}_{n+1}^\ast \mathrm{ch}_k(\CO_{Z_{n+1}}) = \tilde{q}_n^\ast \gamma + \sum_i\tilde{p}_{n+1}^\ast \gamma_i' \cup \tilde{q}_n^\ast \gamma_i'' 
\end{equation}
where 
\[
\gamma, \gamma_i'' \in \bigoplus_l P^{\bar{\pi}^{[n]}}_{l} H^{2l} (S^{[n]}\times S ,\BQ), \quad \gamma'_i \in  \bigoplus_l P^{\bar{\pi}^{[n+1]}}_{l} H^{2l} (S^{[n+1]}\times S ,\BQ).
\]
Applying $\tilde{p}_{{n+1}\ast}$ to (\ref{eqn2018}), the projection formula yields
\begin{equation}\label{eqn22222}
(n+1) \mathrm{ch}_k(\CO_{Z_{n+1}}) = \tilde{p}_{{n+1}\ast} (\tilde{q}_n^\ast \gamma) + \sum_i  \gamma_i' \cup \tilde{p}_{{n+1}\ast}(\tilde{q}_n^\ast  \gamma_i'').
\end{equation}

Finally, by Corollary \ref{functoriality} we have
\[
\tilde{q}_n^\ast \alpha \in P^{\bar{g}_n}_m H^l(S^{[n, n+1]}\times S, \BQ)
\]
for any class $\alpha \in P^{\bar{\pi}_n}_mH^l(S^{[n]}\times S, \BQ)$. Therefore
\[
\tilde{p}_{{n+1}\ast}(\tilde{q}_n^\ast \alpha) \in P^{\bar{\pi}_{n+1}}_m H^l(S^{[n+1]}\times S, \BQ).
\]
In particular, we conclude that the right-hand side of (\ref{eqn22222}) lies in 
\[
P^{\bar{\pi}_{n+1}}_k H^{2k}(S^{[n+1]}\times S , \BQ)
\]
due to the multiplicativity of the perverse filtration. We have completed the induction argument. \qed

\section{Parabolic Higgs bundles and the P=W conjecture}

\subsection{Overview}
In this section, we prove the $P=W$ conjecture for the five families of moduli spaces of parabolic Higgs bundles/local systems studied in \cite[Theorems 4.1 and 5.1]{Gr} and \cite[Section 5]{Z}. Our proof of Theorem \ref{P=W} proceeds by the following steps:
\begin{enumerate}
    \item[1.] We show that there exists a {\it canonical} perverse decomposition on the cohomology $H^\ast(M_1,\BQ)$ defined by the explicit geometry of the morphism $\pi:M_1\rightarrow \BC$. We then prove Theorem \ref{thm0.3} for $n=1$.
    \item[2.] We deduce that the tautological classes are the ring generators of $H^\ast(M_n, \BQ)$. This proves Theorem \ref{thm0.3} (i).
    \item[3.] The perverse decomposition on $H^\ast(M_1,\BQ)$ induces a \emph{canonical} perverse decomposition on $H^\ast(M_n, \BQ)$ for every $n$. Following (a refinement of) the proof of Theorem 0.1, we locate the tautological classes in the perverse decomposition for $M_n$ and complete the proof of Theorem \ref{thm0.3}.  
    \item[4.] We calculate the weights of the tautological classes on the character variety $\CM_n'$. As a consequence, we deduce the decomposition (\ref{weight_decomp}).
     \item[5.] Combining the results above, we prove the $P=W$ conjecture for the moduli spaces $M_n$ and $\CM_n'$ as a match of the decompositions
     \[
     \bigoplus_{k,d} G_kH^d(M_n, \BQ) \xlongequal{P=W} \bigoplus_{k,d}  {^k\mathrm{Hdg}^d(\CM'_n)}.
     \]
\end{enumerate}

\subsection{Step 1: 2D moduli spaces and perverse decompositions}    
The two dimensional moduli spaces associated with affine Dynkin diagrams $\tilde{A}_0$, $\tilde{D}_4$, $\tilde{E}_6$, $\tilde{E}_7$, and $\tilde{E}_8$ are constructed in \cite[Theorem 1.1]{Gr}. See also \cite[Page 669-670]{Z} for a summary. For each Dynkin diagram above, there is a finite group $\Gamma$ acting on an elliptic curve $E$. The induced action on the total cotangent bundle $T^\ast E$ gives a quotient stack $[T^\ast E / \Gamma]$, and the moduli space $M_1$ is given by the crepant resolution of the coarse moduli space of this quotient stack
\begin{equation}\label{resolution}
M_1 \to T^\ast E / \Gamma.
\end{equation}
The cohomology groups of $M_1$ and their perverse filtrations associated with $\pi_1: M_1 \to \BC$ are described in \cite[Sections 5.2 and 5.3]{Z}.

In the $\tilde{A}_0$ case, we have
\[
M_1 \simeq T^\ast E = E \times \BC \xrightarrow{\pi_1} \BC
\]
where $\pi_1$ is the projection to the second factor. The perverse filtration associated with $\pi_1$ admits a natural splitting
\[
H^\ast(M_1, \BQ) = \bigoplus_{i=0}^2 G_iH^i(T^\ast E, \BQ)
\]
with $G_iH^i(T^\ast E, \BQ)=H^i(T^\ast E, \BQ)$.

For $\tilde{D}_4$, $\tilde{E}_6$, $\tilde{E}_7$, and $\tilde{E}_8$, the moduli space $M_1$ is nonsingular with trivial canonical bundle. The restriction of the Hitchin fibration $\pi_1: M_1 \to \BC$ to the nonsingular part of $\pi_1$ is 
\[\
\pi_1: \pi_1^{-1}(\BC^\ast) = E \times \BC^\ast \rightarrow \BC^\ast \subset \BC,
\]
while the closed fiber over $0\in \BC$ has the dual graph given by the corresponding affine Dynkin diagram. All the non-trivial cohomology groups of $M_1$ are
\begin{align*}
    H^0(M_1, \BQ) & =  \langle [M_1] \rangle,\\
    H^2(M_1, \BQ) & =  \langle [s], [E_1], \dots, [E_K] \rangle.
\end{align*}
Here $[s]$ is class of a section of $\pi_1$, $[E_i]$ are the exceptional divisors of the resolution (\ref{resolution}), and
\begin{equation}\label{K}
K= 
\begin{cases}
 4, \quad \tilde{D}_4~~~~~ \textrm{case};\\
 6, \quad \tilde{E}_6~~~~~ \textrm{case};\\
 7, \quad \tilde{E}_7~~~~~ \textrm{case};\\
 8, \quad \tilde{E}_8~~~~~ \textrm{case}.
\end{cases}
\end{equation}
We define the following perverse decomposition associated with $\pi_1$:
\begin{align*}
G_0H^0(M_1, \BQ) &= \langle [M_1] \rangle, \\
G_1H^2(M_1, \BQ) &= \langle [E_1], \dots, [E_K] \rangle, \quad G_2H^2(M_1, \BQ) = \langle [s] \rangle.
\end{align*}

In the $\tilde{A}_0$ case, the group $\Gamma$ is trivial, and 
\[
\BP_\Gamma = E = \CI E.
\]

In the $\tilde{D}_4$, $\tilde{E}_6$, $\tilde{E}_7$, and $\tilde{E}_8$ cases, the (ungraded) cohomology groups of the inertia stacks $\CI \BP_\Gamma$, $\CI[T^\ast E/\Gamma]$, and the moduli space $M_1$ can be canonically identified,
\begin{equation} \label{corresp111}
H_{\mathrm{orb}}^\ast(\BP_\Gamma, \BQ) = H^\ast(\CI[T^\ast E/\Gamma], \BQ) = H^\ast(M_1, \BQ).
\end{equation}
Here the first equality is induced by the projection $\CI[T^\ast E/\Gamma]\rightarrow \CI\BP_\Gamma$ and the second equality is given by the McKay correspondence.\footnote{In this case we can also view the second equality as an identification in the $K$-theory with $\BQ$-coefficients.} By definition, we have
\[
H_{\mathrm{orb}}^\ast(\BP_\Gamma, \BQ) = H^\ast(\BP_\Gamma, \BQ) \oplus \bigoplus_{i=1}^{K} \BQ[e_i],
\]
where every $[e_i] \in H_{\mathrm{orb}}^0(\BP_\Gamma, \BQ)$ is given by an element of the isotropy group of a $\Gamma$-fixed point corresponding to the exceptional divisor $E_i \in H^2(M_1, \BQ)$ via $(\ref{corresp111})$. Assume $(\CE_1, \sigma_1)$ is the universal family on $M_1 \times \CI \BP_\Gamma$ obtained as the pullback of the universal Higgs bundle on $M_1 \times \BP_\Gamma$. We first prove Theorem \ref{thm0.3} in the $n=1$ case.

\begin{prop}\label{prop3.1}
Theorem \ref{thm0.3} holds for $n=1$.
\end{prop}

\begin{proof}
We first treat the $\tilde{A}_0$ case. The universal line bundle $\CE_1$ on $T^\ast E \times E$ is the pullback of the Poincar\'e line bundle 
\[
\CP = \CO_{E\times E}(\Delta - \mathrm{pt} \times E - E \times \mathrm{pt})
\]
from the projection $M_1 \times E \rightarrow E \times E$. Hence the K\"unneth factor of $c_0(\CP)$ on $M_1$ is the fundamental class $[M_1] \in G_0H^0(M_1, \BQ)$, and the K\"unneth factors of $c_1(\CP)$ are the odd degree classes in $G_1H^1(M_1, \BQ)$. They all lie in the correct pieces of the decomposition $G_\bullet$ and generate the total cohomology $H^\ast(M_1, \BQ)$.

Now we consider the $\tilde{D}_4$, $\tilde{E}_6$, $\tilde{E}_7$, and $\tilde{E}_8$ cases, where the $\Gamma$-group actions are nontrivial. We view $\mathrm{ch}(\CE_1)$ as a correspondence on the cohomology
\begin{equation}\label{corr1234}
\mathrm{ch}(\CE_1): H^\ast(\CI\BP_\Gamma, \BQ) \rightarrow H^\ast(M_1, \BQ).
\end{equation}
By the construction of \cite{Gr}, the correspondence (\ref{corr1234}) can be factorized as
\begin{equation}\label{factor}
H^\ast(\CI \BP_\Gamma, \BQ) \xrightarrow{\Phi_1} H^\ast(\CI [T^\ast E/\Gamma], \BQ) \xrightarrow{\Phi_2} H^\ast(M_1, \BQ). 
\end{equation}
Here $\Phi_1$ is given by the $\Gamma$-equivariant Fourier--Mukai functor induced by the Poincar\'e line bundle
\[
\CP: D^b(E) \rightarrow D^b(T^\ast E),
\]
and $\Phi_2$ is an isomorphism given by the McKay correspondence. In particular, we obtain that (\ref{corr1234}) is an isomorphism. This implies that the tautological classes generate $H^\ast(M_1, \BQ)$, and proves the first part of Theorem \ref{thm0.3}.

For the part (ii), a direct calculation using the factorization (\ref{factor}) gives the following equations: 
\begin{align*}
    \Phi_2 \circ \Phi_1 ([\mathrm{pt} \in \BP_\Gamma]) & = [M_1] \in G_0H^0(M_1, \BQ),\\
    \Phi_2 \circ \Phi_1 \langle [e_1], [e_2] , \dots, [e_K] \rangle & = \langle [E_1], [E_2] , \dots, [E_K] \rangle = G_1H^2(M_1, \BQ), \\
    \Phi_2 \circ \Phi_1 ([\BP_\Gamma]) & = [s] \in G_2 H^2(M_1, \BQ).
\end{align*}
By looking at the cohomological degrees, we see that the three identities above are induced by the factors $\mathrm{ch}_0(\CE_1)$, $\mathrm{ch}_1(\CE_1)$, and $\mathrm{ch}_2(\CE_1)$ respectively. In conclusion, the K\"unneth factors of $\mathrm{ch}_k(\CE_1)$ on $M_1$ lies in $G_kH^\ast(M_1, \BQ)$. This completes the proof of the proposition in the $\tilde{D}_4$, $\tilde{E}_6$, $\tilde{E}_7$, and $\tilde{E}_8$ cases.
\end{proof}

\subsection{Step 2: tautological classes on $M_n$}  We show in this section that the tautological classes 
\begin{equation}\label{eqn34}
\int_{\alpha} c_k(\CE) \in H^\ast(M_n, \BQ), \quad \alpha \in H_{\mathrm{orb}}^\ast(\BP_\Gamma, \BQ)
\end{equation}
generate the cohomology ring $H^\ast(M_n, \BQ)$. It is parallel to a result of Markman \cite{Markman} in the case of (non-parabolic) Higgs bundles on a compact Riemann surface of genus $\geq 2$. However, our proof relies on the structure for the cohomology of the Hilbert scheme of points on a projective surface \cite{LQW}, which is different with the approach in \cite{Markman}.

\begin{proof}[Proof of Theorem \ref{thm0.3} (i)]
{\bf 1. The $\tilde{A}_0$ case.} We first treat the $\tilde{A}_0$ case. For a projective nonsingular surface $S$, recall that the result of Li--Qin--Wang \cite{LQW} implies that the image of the correspondence 
\begin{equation*}
    \mathrm{ch}(\CO_{Z_n}): H^\ast(S, \BQ) \rightarrow H^\ast(S^{[n]}, \BQ)
\end{equation*}
generates the \emph{ring} $H^\ast(S^{[n]}, \BQ)$, where $Z_n$ is the universal subscheme. For our purpose, we take a compactification of $M_1$,
\[
M_1 = T^\ast E \subset E\times \BP^1 = \overline{M}_1,
\]
which induces a natural compactification of $M_n$,
\[
M_n \subset \overline{M}_n=\overline{M}_1^{[n]}.
\]
The restriction map on the cohomology groups
\begin{equation}\label{restr}
H^\ast(\overline{M}_n, \BQ) \xrightarrow{\mathrm{res.}} H^\ast(M_n, \BQ)
\end{equation}
is surjective by the decomposition (\ref{GS_decomp}). We define ${\Xi}'_n \subset \overline{M}_n \times \overline{M}_1$ to be the be the universal subscheme, and let ${\Xi}_n$ be its restriction on $M_n \times \overline{M}_1$. 

By the construction of \cite{Gr}, the derived functor $D^b(E) \to D^b(M_n)$ given by the universal family $\CE_n \in D^b(M_n\times E)$ can be factorized as
\begin{equation}\label{factor31}
D^b(E) \to D^b(M_1) \to D^b(M_n)
\end{equation}
where the first morphism is induced by the universal family $\CE_1$ on $M_1 \times E$, and the second morphism is induced by the structure sheaf of the universal subscheme in $M_n \times M_1 = M_1^{[n]}\times M_1$.\footnote{Although $M_1$ is not proper, the Fourier--Mukai transform $D^b(M_1) \to D^b(M_n)$ is well-defined. This is because the support of $\CO_{\Xi_n}$ in $M_n \times M_1$ is proper over $M_n$.}

We consider the following commutative diagram of correspondences for cohomology groups,
\begin{equation*}
\begin{tikzcd}
H^\ast(E, \BQ) \ar [rr,"\phi_n"]\ar[rd,"\bar{\phi}_1"]& & H^\ast(M_n, \BQ)\\
 &H^\ast(\overline{M}_1, \BQ) \ar[ru,"g_n"]&. 
\end{tikzcd}
\end{equation*}  
Here $\phi_n$ is induced by the functor (\ref{factor31}), and $\bar{\phi}_1$ is induced by the pullback of the Poincar\'e line bundle on $E \times E$ via the natural projection $\overline{M}_1 \to E\times E$, and $g_n$ is induced by $\CO_{{\Xi}_n} \in \mathrm{Coh}(M_n \times \overline{M}_1)$.\footnote{Here the diagram is commutative since the Todd class of $\overline{M}_1$ is supported on the divisor $\overline{M}_1\backslash M_1$.} For the proof of Theorem \ref{thm0.3} (i), it suffices to show that the image of $\phi_n$ generates the cohomology ring $H^\ast(M_n, \BQ)$.

By \cite{LQW} and the surjectivity of the restriction map (\ref{restr}), we observe that the image of $g_n$ generates $H^\ast(M_n, \BQ)$. Moreover, a direct calculation implies that any class in the image of $\bar{\phi}_1$ does not intersect with the divisor 
\[
E \times \{\infty\} = \overline{M}_1 \backslash M_1.
\]
Since the support of the class $\mathrm{ch}(\CO_{{\Xi}_n})$ does not intersect with the locus
\[
(\overline{M}_1 \backslash M_1) \times M_n,
\]
we obtain from the projective bundle formula that 
\[
H^\ast(\overline{M}_1, \BQ) = H^\ast(M_1) \oplus N = \mathrm{Im}(\bar{\phi}_1)\oplus N
\]
where $N$ are the classes supported on $\overline{M}_1\backslash M_1$, and therefore $N \subset  \mathrm{Ker}(g_n)$. As a consequence, we have 
\[
\mathrm{Im}(\phi_n) = \mathrm{Im}(g_n).
\]
This completes the proof in the $\tilde{A}_0$ case.

{\bf 2. The other cases.} The cases of $\tilde{D}_4$, $\tilde{E}_6$, $\tilde{E}_7$, and $\tilde{E}_8$ are similar. The action of $\Gamma$ on an elliptic curve $E$ can be lifted to the total space of the cotangent bundle $T^\ast E$. We consider the induced $\Gamma$-action on the surface $E\times \BP^1$ (viewed as a compactification of the open surface $T^\ast E$). Since $M_1$ is the crepant resolution of the singular surface $T^\ast E/\Gamma$, a nonsingular resolution of the quotient $E\times \mathbb{P}^1/\Gamma$, denoted by $\overline{M}_1$, is then a compactification of $M_1$ with the surjective restriction map
\begin{equation}\label{eqn}
H^\ast(\overline{M}_1, \BQ) \xrightarrow{\mathrm{res.}} H^\ast(M_1, \BQ).
\end{equation}
Geometrically $\overline{M}_1$ contains $M_1$ as an open sub-surface and the boundary is
\[
\overline{M}_{1,{\infty}} = \cup_i^l D_i \cup F
\]
with $D_i$ the exceptional curves. Hence we have
\begin{equation}\label{eqnn35}
H^\ast(\overline{M}_1,\BQ) = H^\ast(M_1, \BQ) \oplus N, \quad N= \langle [D_1], \dots, [D_l], [F], [\mathrm{pt}] \rangle. 
\end{equation}

Similar to the $\tilde{A}_0$ case, we consider the Hilbert scheme $\overline{M}_n= \overline{M}_1^{[n]}$ as a nonsingular compactification of $M_n$. By the surjectivity of (\ref{eqn}) and the decomposition (\ref{GS_decomp}), we obtain that the restriction map 
\[
H^\ast(\overline{M}_n, \BQ) \xrightarrow{\mathrm{res.}} H^\ast(M_n, \BQ)
\]
is also surjective. Since $M_n$ is the moduli space of $\Gamma$-equivariant stable parabolic Higgs bundles on the elliptic curve $(E, o_E)$ \cite{Gr}, the derived functor $D^b([E/\Gamma]) \to D^b(M_n)$ induced by the universal family can be factorized as
\begin{equation}\label{factorization}
D^b([E/\Gamma]) \to D^b(M_1) \to D^b(M_n)
\end{equation}
similar as (\ref{factor31}).

On the level of cohomology, we consider the following commutative diagram of correspondences,
\begin{equation*}
\begin{tikzcd}
H^\ast(\CI\BP_\Gamma, \BQ) \ar [rr,"\phi_n"]\ar[rd,"\bar{\phi}_1"]& & H^\ast(M_n, \BQ)\\
 &H^\ast(\overline{M}_1, \BQ) \ar[ru,"g_n"]&. 
\end{tikzcd}
\end{equation*}  
Here $\phi_n$ is induced by the pullback of the universal family on $M_n\times [E/\Gamma]$ to $M_n\times \CI\BP_\Gamma$, the morphism
\[
\bar{\phi}_1: H^\ast(\CI\BP_\Gamma, \BQ) \rightarrow H^\ast(\CI[T^\ast E/\Gamma], \BQ)=H^\ast(M_1, \BQ) \hookrightarrow H^\ast(\overline{M}_1, \BQ)
\]
is induced by the $\Gamma$-equivariant universal Poincar\'e sheaf on $E\times T^\ast E$ (see Proposition \ref{prop3.1}), and $g_n$ is induced by the structure sheaf $\CO_{\Xi_n}$.

We obtain from (\ref{eqnn35}) that
\[
H^\ast(\overline{M}_1, \BQ) = H^\ast(M_1, \BQ) \oplus N = \mathrm{Im}(\bar{\phi}_1) \oplus N,
\]
where $N$ is generated by the classes supported on $\overline{M}_1 \backslash M_1$. The same proof as in the $\tilde{A}_0$ case implies that
\[
N \subset  \mathrm{Ker}(g_n).
\]
Hence $\mathrm{Im}(\phi_n)= \mathrm{Im}(g_n)$. In particular, we conclude that the image of $\phi_n$ generates $H^\ast(M_n, \BQ)$.
\end{proof}

\subsection{Step 3: perverse decompositions and tautological classes}
By Section 2, the decomposition
\begin{equation}\label{eqn2020}
H^\ast(M_1, \BQ) = \bigoplus_{i=0}^2 G_iH^\ast(M_1, \BQ)
\end{equation}
of Section 3.2 induces the perverse decomposition (\ref{perverse_decomp}) for every moduli space $M_n= M_1^{[n]}$. The purpose of this section is to calculate the location of the tautological classes
(\ref{eqn34}) in the perverse decomposition of $H^\ast(M_n, \BQ)$ and complete the proof of Theorem \ref{thm0.3}.

\begin{lemma}\label{lem3.2}
The perverse decomposition $G_\bullet H^\ast(M_n, \BQ)$ is strongly multiplicative, \emph{i.e.},
\[
G_iH^\ast(M_n,\BQ)\cup G_jH^\ast(M_n,\BQ)\subset G_{i+j}H^\ast(M_n,\BQ).
\]
\end{lemma}

\begin{proof}
It suffices to check that the conditions of Proposition \ref{Prop1.8} hold for the surface $M_1$. The conditions ($\dagger$) (1) and (ii) are clearly satisfied, and ($\dagger$) (2) follows from \cite[Proposition 0.5]{Z}.
\end{proof}

Let $\pi': \overline{M}_1 \rightarrow \BP^1$ be the natural projection which compactifies the Hitchin fibration $\pi_1: M_1 \to \BC$. By the description (\ref{eqnn35}), the perverse filtration associated with $\pi'$ admits a canonical splitting 
\begin{equation}\label{splitting12}
H^\ast(\overline{M}_1, \BQ) = \bigoplus_{i=0}^2 G_iH^\ast(\overline{M}_1, \BQ)
\end{equation}
satisfying that 
\[
[\textup{pt}]\in G_2H^4(\overline{M}_1, \BQ),~~~~~~ [F]\in G_0H^2(\overline{M}_1,\BQ),~~~~~~  [D_i] \in G_1H^2(\overline{M}_1, \BQ),
\]
and its restriction to $H^\ast(M_1, \BQ)$ coincides with (\ref{eqn2020}).\footnote{Here we omit the discussion for the $\widetilde{A}_0$ case, since the decomposition (\ref{splitting12}) is obviously obtained for $\pi':\overline{M}_1= E\times \BP^1 \to \BP^1$.}

For notational convenience, if the cohomology of each variety $X_i$ admits a perverse decomposition $G_\bullet H^\ast(X_i, \BQ)$ with respect to a morphism $X_i\to Y_i$, we define 
\begin{multline}
    G_kH^\ast(X_1\times X_2\times \dots \times X_m, \BQ)\\
    =\left\langle \alpha_1\boxtimes \dots \boxtimes \alpha_m; ~~~ \alpha_i \in G_{k_i}H^\ast(X_i, \BQ),~~~ \sum_{i}k_i=m \right\rangle,
\end{multline}
which gives a perverse decomposition for $X_1\times X_2\times \dots \times X_m$ with respect to the Cartesian product
\[
X_1\times X_2\times \dots \times X_m \to Y_1\times Y_2\times \dots \times Y_m
\]
by Proposition \ref{kunneth}. In particular (\ref{splitting12}) induces a perverse decomposition for $M_n \times \overline{M}_1$ with respect to the morphism
\[
\pi_n \times \pi' : M_n \times \overline{M}_1 \to \BC^n \times \BP^1.
\]

Recall that $\Xi_n$ is the restricted universal subscheme of $M_n \times \overline{M}_1$. The following theorem is parallel to Theorem \ref{thm1}. We show it by the induction scheme of Section 2.3.

\begin{thm}\label{Thm3.3}
We have
\[
\mathrm{ch}_k(\CO_{\Xi_n}) \in G_k H^{2k}(M_n \times \overline{M}_1, \BQ).
\]
\end{thm}

\begin{proof}
We first verify the theorem for $n=1$.  The subscheme $Z_1\subset M_1 \times \overline{M}_1$ is the restriction of the diagonal
\[
\Delta \subset \overline{M_1} \times \overline{M_1}.
\]
A direct calculation of its class using the basis (\ref{eqnn35}) implies that
\[
[\Delta] \in G_2H^4(\overline{M_1} \times \overline{M_1}, \BQ).
\]
Hence we obtain via the restriction map (\ref{eqn}) that
\[
c_2(\CO_{Z_1}) \in G_2H^4({M_1} \times \overline{M_1}, \BQ).
\]
Now since 
\[
c_1(\CO_{Z_1}) = c_4(\CO_{Z_1}) =0,
\]
it suffices to show that $c_3(\CO_{Z_1}) \in G_3H^6({M_1} \times \overline{M_1}, \BQ)$. 

In fact, by the Grothendieck--Riemann--Roch formula, the class $c_3(\CO_\Delta)$ is proportional to the class
\[
\mathrm{pr}_1^\ast c_1(\overline{M}_1) \cup [\Delta]  \in H^6(\overline{M_1} \times \overline{M_1}, \BQ).
\]
Its restriction to ${M_1} \times \overline{M_1}$ vanishes due to the fact that $c_1(\overline{M}_1)$ is supported on $\overline{M_1} \backslash M_1$. 

The rest of the proof is exactly the same as the induction argument in Section 2.3. The only minor difference is that, instead of considering the perverse filtration $P_\bullet H^\ast(M_n \times \overline{M}_1)$, we work with the perverse decomposition $G_\bullet H^\ast(M_n \times \overline{M}_1)$. The corresponding comparison results for the pushforward morphism along
\[
M_1^{[n,n+1]} \to M_{n+1}
\]
and the pullback morphism along
\[
M_1^{[n,n+1]} \to M_n \times M_1
\]
are given by Propositions \ref{prop1.4} and \ref{prop1.5}. The precise location of the exceptional divisor in the decomposition $G_\bullet H^\ast(M_n, \BQ)$ is calculated in Lemma \ref{perv_exc}. Finally, the \emph{multiplicativity} of the perverse filtration used in Section 2.3 is replaced by the \emph{strong multicplicativity} for the decomposition $G_\bullet H^\ast(M_n, \BQ)$ (Lemma \ref{lem3.2}).
\end{proof}

For convenience, we introduce the trivial perverse decomposition on the inertia orbifold $\CI \BP_\Gamma$,
\[
G_0 H^\ast(\CI \BP_\Gamma, \BQ) = H^\ast(\CI \BP_\Gamma. \BQ).
\]
By the K\"unneth decomposition, it induces the perverse decompositions $G_\bullet H^\ast(M_n \times \CI\BP_\Gamma, \BQ)$ and $G_\bullet H^\ast(M_n \times \overline{M}_1\times \CI\BP_\Gamma, \BQ)$ which split the perverse filtrations associated with the morphisms
\[
M_n \times \CI\BP_\Gamma \to \BC^n \times \CI\BP_\Gamma
\]
and
\[
M_n \times \overline{M}_1\times \CI\BP_\Gamma \to \BC^n \times \BP^1 \times \CI\BP_\Gamma.
\]
We define the class
\begin{equation*}
    I(\alpha, \gamma) = \int_{\gamma} \mathrm{pr}_{12}^\ast \mathrm{ch}(\CO_{\Xi_n}) \cup \mathrm{pr}_{23}^\ast \alpha \in H^\ast(M_n \times \CI\BP_\Gamma, \BQ)
\end{equation*}
with $\gamma \in H^\ast(\overline{M}_1, \BQ)$, $\alpha \in H^\ast(\overline{M}_1\times \CI\BP_\Gamma, \BQ)$, and $\mathrm{pr}_{ij}$ the projections. Since the Chern character $\mathrm{ch}(\CO_{\Xi_n})$ is supported on the open subset
\[
M_n \times M_1 \subset M_n \times \overline{M}_1,
\]
the class $I(\alpha, \gamma)$ only depends on the restriction of $\alpha$ to $H^\ast(M_1\times \CI\BP_\Gamma, \BQ)$ by the K\"unneth decomposition. Hence $I(\alpha, \gamma)$ is well-defined for \[\alpha \in H^\ast(M_1\times \CI\BP_\Gamma, \BQ).
\]

Let $(\CE_n ,\theta_n)$ denote the universal family on $M_n\times \BP_\Gamma$. For our purpose, it suffices to show that\footnote{For notational convenience, we also use $\mathrm{ch}(\CE_n)$ to denote the pullback of its corresponding class on $M_n\times \BP_\Gamma$ to $M_n\times \CI\BP_\Gamma$.}
\begin{equation}\label{mainformula}
\mathrm{ch}(\CE_n) \in \bigoplus_i G_i H^{2i}(M_n\times \CI\BP_\Gamma, \BQ).
\end{equation}
Since the derived functor induced by the universal family $\CE_n$ on $M_n\times \BP_\Gamma$ is factorized as (\ref{factorization}), we obtain that (\ref{mainformula}) is equivalent to the following condition:
\begin{equation}\label{Integral}
I(\mathrm{ch}(\CE_1), [\overline{M}_1]) \in \bigoplus_i G_i H^{2i}(M_n\times \CI\BP_\Gamma, \BQ).
\end{equation}

In the following, we prove Theorem \ref{thm0.3} (ii) by verifying (\ref{Integral}).

\begin{proof}[Proof of Theorem \ref{thm0.3} (i)]
Assume a variety $X$ admits a perverse decomposition $G_\bullet H^\ast(X, \BQ)$. We call a class $\alpha \in H^\ast(X, \BQ)$ \emph{balanced}, if 
\[
\alpha \in \bigoplus_i G_iH^{2i}(X, \BQ).
\]
By the calculation in the proof of Proposition \ref{prop3.1}, the class
\[
\mathrm{ch}(\CE_1) \in H^\ast(M_1 \times \CI\BP_\Gamma, \BQ)
\]
is balanced. The Chern character $\mathrm{ch}(\CO_{\Xi_n})$ is also balanced by Theorem \ref{Thm3.3}. 

Since $[M_n] \in G_0H^0(M_n,\BQ)$ and $[\CI\BP_\Gamma]\in G_0H^0(\CI\BP_\Gamma,\BQ)$, Proposition \ref{kunneth} implies that the pullbacks $\mathrm{pr}^\ast_{12}$ and $\mathrm{pr}^\ast_{23}$ preserve the perverse decompositions. Hence we deduce from Lemma \ref{lem3.2} that the class
\[
\mathrm{pr}_{12}^\ast \mathrm{ch}(\CO_{\Xi_n}) \cup \mathrm{pr}_{23}^\ast \mathrm{ch}(\CE_1)
\]
is balanced.\footnote{We view $\mathrm{ch}(\CE_1) \in H^\ast({M}_1\times \CI\BP_\Gamma, \BQ)$ as a class in $H^\ast(\overline{M}_1\times \CI\BP_\Gamma, \BQ)$ by the K\"unneth decomposition and $H^\ast(\overline{M}_1, \BQ)= H^\ast({M}_1, \BQ)\oplus N$.} Note that the integration over the fundamental class $[\overline{M}_1]$ is equivalent to picking up the K\"unneth factor corresponding to the point class \[
[\mathrm{pt}] \in G_2H^4(\overline{M}_1, \BQ).
\]
We find that $I(\mathrm{ch}(\CE_1), [\overline{M}_1])$ is balanced.
\end{proof}

\subsection{Step 4: weight filtrations and tautological classes}
In this section we compute the weights of the tautological classes following the method in \cite{Shende}. Let $\CM'_n$ denote the parabolic character variety corresponding to the parabolic Higgs moduli $M_n$ via the nonabelian Hodge theorem; see \cite[Section 5.1]{GNR}, \cite[Theorem 5.1]{Gr}, and \cite[Section 5.1]{Z} for details. By the parabolic nonabelian Hodge theorem \cite[Theorem 7.10]{B2}, the character variety $\CM'_n$ is canonically homeomorphic to $M_n$. The following result is parallel to Theorem \ref{thm0.3} for Higgs bundles.

\begin{thm}\label{thm3.4}
Let $\CE'_n$ be the universal family on $\CM_n'\times \CI\BP_\Gamma$ which is obtained as the pullback of the universal family on $\CM_n'\times\BP_\Gamma$.\footnote{Here the universal family is given by the flat bundles corresponding to the local systems via the Riemann--Hilbert correspondence.} The following statements hold:
\begin{enumerate}
    \item[(i)] The tautological classes
    \[
    \int_\alpha{\mathrm{c}_k(\CE'_n)}\in H^\ast(\CM'_n,\BQ),\quad \alpha \in H_{\mathrm{orb}}^\ast(\BP_\Gamma, \BQ)
    \]
    generate the cohomology ring $H^\ast(\CM_n',\BQ)$.
    \item[(ii)] We have
    \[
    \int_\alpha{\mathrm{c}_k(\CE'_n)}\in {^k\mathrm{Hdg}^{\ast}}(\CM'_n,\BQ)
    \]
    for any $\alpha\in H^\ast_{\mathrm{orb}}(\BP_\Gamma,\BQ)$.
\end{enumerate}

\end{thm}

\begin{proof}
Recall from \cite{Gr} that the moduli space $M_n$ parametrizes all triples $(\CF,\theta,v)$, where $(\CF,\theta)$ is a semistable degree 0 Higgs bundle on the orbifold $\BP_\Gamma$ and $v$ is a vector in the fiber $\CF_0$. Here the stability condition refers to that any Higgs sub-bundle containing $v$ has negative degree. 

By forgetting parabolic structures of the objects parametrized by $M_n$, we obtain a morphism $M_n\rightarrow \CH iggs_n(\BP_\Gamma)$, where $\CH iggs_n(\BP_\Gamma)$ denotes
the moduli stack of degree 0 rank $n$ Higgs bundles on $\BP_\Gamma$, or equivalently, the moduli stack of $\Gamma$-equivariant Higgs bundles on $E$ of degree 0 and rank $n$. 

%$\mathrm{GL}_n$-representations of $\pi_1^{orb}(\BP_\Gamma)$.\footnote{See \cite[Section 1A]{NS} for definition of orbifold fundamental groups.}

Similarly, we denote by $\CL oc(\BP_\Gamma)$ the moduli stack of $\Gamma$-equivariant rank $n$ local systems on $E$. Let $\Delta_{\CI\BP_\Gamma}$ be a simplicial resolution of the inertia stack $\CI\BP_\Gamma$. We have the following commutative diagrams. 

\[
\begin{tikzcd}
\CM_n'\times \Delta_{\CI\BP_\Gamma}\arrow[r, "homotopic"]\arrow{d} & \CM_n'\times \CI\BP_\Gamma\arrow[r,"homeomorphic"]\arrow[d]&  M_n\times\CI\BP_\Gamma\arrow[d]\\
\CL oc_n(\BP_\Gamma)\times \Delta_{\CI\BP_\Gamma}\arrow{r}\arrow{dr}&\CL oc_n(\BP_\Gamma)\times \CI\BP_\Gamma\arrow[d,"cont."]&  \CH iggs_n(\BP_\Gamma)\times \CI\BP_\Gamma\arrow[dl,"cont."]\\
 &\mathrm{BGL}_n &
\end{tikzcd}
\]
%bundle on $\CB un_n(\BP_\Gamma)\times \BP_\Gamma$ and the universal principal $\mathrm{GL}_n$-bundle on $\CL oc_n(\BP_\Gamma)\times \BP_\Gamma$.
Here the canonical homeomorphism on the top is induced by the parabolic nonabelian Hodge theorem, and the two continuous maps are induced by the universal families. It follows immediately from the construction that $\CE_n'$ (resp. $\CE_n$) is the pullback of $\mathrm{EGL}_n$ from $\mathrm{BGL}_n$ to $\CM'_n\times \CI\BP_\Gamma$ (resp. $M_n\times \CI\BP_\Gamma$). Therefore, the tautological classes are identified via the canonical homeomorphism $\CM'_n \cong M_n$. In particular, the statement (i) follows directly from the identification \[
H^\ast(\CM'_n,\BQ)=H^\ast(M_n,\BQ)
\]
and Theorem \ref{thm0.3} (i).

Now we prove (ii).
The argument in \cite{Shende} proves that the morphism
\[
\CL oc_n(\BP_\Gamma)\times \Delta_{\CI\BP_\Gamma}\rightarrow \textrm{BGL}_n
\]
is algebraic. Note that all the top vertical arrows in the diagrams are forgetful maps, which are also algebraic. Hence we obtain the algebraicity of the morphism
\begin{equation}\label{ALG}
\mathrm{ev}: \CM_n' \times \Delta_{\CI\BP_\Gamma}\rightarrow \textrm{BGL}_n.
\end{equation}
%$\CL oc_n(\BP_\Gamma)$ is an Artin stack, it can be viewed as a simplicial scheme over $\BC$. Recall that

Since simplicial schemes carry mixed Hodge structures which are functorial with respect to algebraic morphisms \cite[Definition 8.3.4, Proposition 8.3.9]{D}, we get the following morphism of mixed Hodge structures from (\ref{ALG}),
\[
H^\ast(\mathrm{BGL}_n,\BQ)\rightarrow H^\ast(\CM'_n,\BQ)\otimes H^\ast(\Delta_{\CI\BP_\Gamma},\BQ).
\]

The cohomology of $\mathrm{BGL}_n$ was calculated in \cite[Theorem 9.1.1]{D},  
\[
H^\ast(\mathrm{BGL}_n,\BQ)=\bigoplus{^k\mathrm{Hdg}^{2k}(\mathrm{BGL}_n)},
\]
with ring generators given by the tautological classes
\[
\mathrm{c}_k(\mathrm{EGL}_n)\in {^k\mathrm{Hdg}^{2k}(\mathrm{BGL}_n)}.
\]
Since $H^\ast(\Delta_{\CI\BP_\Gamma},\BQ)$ is isomorphic to $H^\ast_{\mathrm{orb}}(\BP_\Gamma, \BQ)$ as a $\BQ$-vector space and carries the Hodge structure in which all classes have weight 0, we have
\[
\int_\alpha{\mathrm{c}_k(\CE'_n)} = \int_\alpha \mathrm{ev}^\ast \mathrm{c}_k(\mathrm{EGL}_n)  \in {^k\mathrm{Hdg}^{\ast}}(\CM'_n,\BQ)
\]
for any $\alpha\in H^\ast_{\mathrm{orb}}(\BP_\Gamma,\BQ)$.
\end{proof}

As a corollary, we have the following structural result for the cohomology of $\CM_n'$ which proves (\ref{weight_decomp}).

\begin{cor}
The cohomology group of $\CM_n'$ admits a canonical bi-grading decomposition
\[
H^\ast(\CM'_n, \BQ) = \bigoplus_{k,d}  {^k\mathrm{Hdg}^d(\CM'_n)}.
\]
\end{cor}

\subsection{Step 5: the P=W conjecture}

\begin{proof}[Proof of Theorem \ref{P=W}]
By Theorems \ref{thm0.3} and \ref{thm3.4}, the tautological classes are the generators of the cohomology ring
\[
H^\ast(\CM'_n,\BQ)=H^\ast(M_n,\BQ),
\]
and they lie in the same piece of the decompositions (\ref{perverse_decomp}) and (\ref{weight_decomp}). Hence Lemma \ref{lem3.2} yields the canonical isomorphism of the decompositions
\[
 \bigoplus_{k,d} G_kH^d(M_n, \BQ) \xlongequal{P=W} \bigoplus_{k,d}  {^k\mathrm{Hdg}^d(\CM'_n)}
 \]
 via the non-abelian Hodge theorem.
\end{proof}

\subsection{Applications: mixed Hodge numbers of character varieties}
For any complex algebraic variety $X$, we define the mixed Hodge polynomial of $X$ as
\[
P(X;x,y,t)=\sum_{p,q,d}\dim (\mathrm{Gr}^W_{p+q}H^d(X,\BC))^{p,q}x^py^qt^d.
\]    
The following result follows immediately from Theorem \ref{P=W} and \cite[Theorem 5.11]{Z}, which concerns the generating series of $P(\CM_n';x,y,t)$. 
\begin{thm}   
Let $q=xy$. In the $\tilde{A}_0$ case, we have
\[
\sum_{n=0}^\infty s^nP(\CM'_n;x,y,t)=\prod_{m=1}^\infty\frac{(1+s^mq^{m}t^{2m-1})^2}{(1-s^mq^{m-1}t^{2m-2})(1-s^mq^{m+1}t^{2m})}.
\]
In the $\tilde{D}_4$, $\tilde{E}_6$, $\tilde{E}_7$, or $\tilde{E}_8$ cases, we have
\begin{multline*}
\sum_{n=0}^\infty s^nP(\CM'_n;x,y,t)\\
=\prod_{m=1}^\infty\frac{1}{(1-s^mq^{m-1}t^{2m-2})(1-s^mq^mt^{2m})^K(1-s^mq^{m+1}t^{2m})},
\end{multline*}
where $K$ is defined in (\ref{K}).
\end{thm}

In general, mixed Hodge numbers for character varieties are difficult to compute, since the description of a character variety as an affine GIT quotient is not adapted to the computation of mixed Hodge structures. As a corollary of the $P=W$ identity, we are able to obtain closed formulas for the mixed Hodge numbers in the five families of parabolic character varieties using the corresponding perverse filtrations. Our result matches the conjecture proposed by Hausel, Letellier, and Rodriguez-Villegas \cite[Conjecture 1.2.1]{HLR} for the first few values of $n$.\footnote{This is verified numerically by the computer program. However there are still combinatorial difficulties to prove that our formulas coincide with the conjecture in \cite{HLR} as closed formulas; see \cite[Remark 5.13]{Z}.}


\begin{thebibliography}{10}

\bibitem{BBD} A. A. Beilinson, J. Bernstein, P. Deligne, {\em Faisceaux pervers,} in Analysis and Topology on Singular Spaces, I (Luminy, 1981), Ast\'erisque 100, Soc. Math. France, Paris, 1982, pp. 5--171.

\bibitem{B1} O. Biquard, {\em Fibr\'es de Higgs et connexions int\'egrables: le cas logarithmique (diviseur lisse),} Ann. Sci. \'Ecole Norm. Sup. (4) 30 (1997), no. 1, 41--96. 

\bibitem{B2} O. Biquard, O. Garc\'ia-Prada, I. Mundet i Riera, {\em Parabolic Higgs bundles and representations of the fundamental group of a punctured surface into a real group,} arXiv:1510.04207.


\bibitem{Ch} J. Cheah, {\em Cellular decompositions for nested Hilbert schemes of points,} Pacific
J. Math. 183 (1998), no. 1, 39--90.


%\bibitem{dCM3} M. de Cataldo, L. Migliorini, {\em The Douady space of a complex surface,} Adv. in Math. 151 (2000), 283--312.


\bibitem{dCM0} M. de Cataldo, L. Migliorini, {\em The Hodge theory of algebraic maps,} Ann. Sci. \'Ecole Norm. Sup. 38 (2005), 693--750.


\bibitem{dCM1} M. de Cataldo, L. Migliorini, {\em The decomposition theorem, perverse sheaves and the topology of algebraic maps,} Bull. Amer. Math. Soc. 46 (2009), 535--633.

\bibitem{dCM2} M. de Cataldo, L. Migliorini, {\em Intersection forms, topology of maps and motivic decomposition for resolutions of threefolds,} Algebraic Cycles and Motives, London Math. Soc. Lecture Note Series,
n.343, vol.1, pp.102-137, Cambridge University Press, Cambridge, UK, 2007.

\bibitem{dCHM3} M. de Cataldo, T. Hausel, L. Migliorini, {\em Exchange between perverse and weight filtration for the Hilbert schemes of points of two surfaces,}
Journal of Singularities, vol 7 (2013), 23--38.

\bibitem{dCM3} M. de Cataldo, L. Migliorini, {\em The Chow groups and the motives of the Hilbert scheme of points on a surface,} Journal of Algebra 251, 824-848 (2002).

\bibitem{dCM4} M. de Cataldo, L. Migliorini, {\em The Chow motive of semismall resolutions,} Mathematical Research Letters 11, 151-170 (2004).

\bibitem{dCHM1} M. de Cataldo, T. Hausel, L. Migliorini, {\em Topology of Hitchin systems and Hodge theory of character varieties: the case $A1$,} Annals of Mathematics 175 (2012), 1329--1407.

\bibitem{D} P. Deligne, {\em Theorie de Hodge, III}, Publications Math\'ematiques de l'IHES 44.1 (1974): 5-77.

\bibitem{Go1} L. G\"ottsche, {\em The Betti numbers of the Hilbert scheme of points on a smooth projective surface,} Math. Ann. 286 (1990), 193--207.

\bibitem{Go2} L. G\"ottsche, W. S\"orgel, {\em Perverse sheaves and the cohomology of Hilbert schemes of smooth algebraic surfaces,} Math. Ann. 296 (1993), 235--245.

\bibitem{GNR} A. Gorsky, N. Nekrasov, V. Rubtsov, {\em Hilbert schemes, separated variables, and D-branes,} Comm. Math. Phys. 222 (2001), no. 2, 299--318.

\bibitem{Gr} M. Gr\"ochenig, {\em Hilbert schemes as moduli of Higgs bundles and local systems,} Int. Math. Res. Not. (2014) 2014 (23): 6523--6575.

\bibitem{HLR} T. Hausel E. Letellier, F. Rodriguez-Villegas, {\em Arithmetic harmonic analysis on character and quiver varieties.} Duke Math. J. 160 (2011), no. 2, 323--400.

\bibitem{I} M. Inaba, M. Saito, {\em Moduli of unramified irregular singular parabolic connections on a smooth projective curve,} Kyoto J. Math. 53, no. 2 (2013), 433-482.

\bibitem{L} M. Lehn, {\em Chern classes of tautological sheaves on Hilbert schemes of points on surfaces,} Inventiones Mathematicae, 1999, 136(1):157--207.

\bibitem{LS} M. Lehn, C. Sorger, {\em The cup product of the Hilbert scheme for K3 surfaces,} Inventiones mathematicae May 2003, Volume 152, Issue 2, pp 305-329.


\bibitem{LQW} W. Li, Z. Qin, W. Wang, {\em Vertex algebras and the cohomology ring structure of Hilbert schemes of points on surfaces,} Math. Ann. 324 (2002),
105--133.


\bibitem{Markman} E. Markman, {\em Generators of the cohomology ring of moduli spaces of sheaves on symplectic surfaces,} J. Reine Angew. Math. 544 (2002), 61--82. 


\bibitem{Na} H. Nakajima, {\em Heisenberg algebra and Hilbert schemes of points on projective
surfaces,} Ann. Math. 145 (1997), 379--388.

%\bibitem{NS} B. Nasatya, B. Steer, {\em Orbifold Riemann surfaces and the Yang-Mills-Higgs equations,} Annali della Scuola Normale Superiore di Pisa - Classe di Scienze 22.4 (1995): 595-643.


\bibitem{Shende} V. Shende, {\em The weights of the tautological classes of character varieties,} Int. Math. Res. Not. 2017, no. 22, 6832--6840.

\bibitem{Simp90} C. T. Simpson, {\em Harmonic bundles on noncompact curves,} J. Amer. Math. Soc. 3 (1990), 713--770.

\bibitem{Simp} C. T. Simpson, {\em Higgs bundles and local systems,} Publ. Math. Inst. Hautes Etudes Sci. 75 (1992) 5--95.

\bibitem{Z} Z. Zhang, {\em Multiplicativity of perverse filtration for Hilbert schemes of fibered surfaces,} Adv. Math. 312 (2017) 636--679.


\end{thebibliography}
\end{document}